\documentclass[10pt,a4paper]{amsart}
\usepackage{hyperref}
\usepackage{mathrsfs}
\usepackage{amsmath}
\usepackage{amssymb}
\usepackage[all]{xy}
\usepackage{latexsym}
\usepackage[T1]{fontenc}
\usepackage{datetime}
\usepackage{pgf,tikz}
\usetikzlibrary{arrows}

 \usepackage[latin1]{inputenc}

\newcommand{ \inj}{ \hookrightarrow}
\newcommand{ \surj}{ \twoheadrightarrow}
\newcommand{ \Hom}{ \mathrm{Hom}}

\newcommand{\R}{\mathbb{R}}
\newcommand{\Z}{\mathbb{Z}}
\newcommand{\F}{\mathbb{F}_2}

\newcommand{\mf}{\underline{\F}}
\newcommand{\mz}{\underline{\Z}}
\newcommand{\ste}{\mathcal{A}}

\newcommand{\Ab}{\mathcal{A}b}
\newcommand{\ab}{\mathcal{B}}

\newcommand{\sh}{\mathcal{SH}}
\newcommand{\Sp}{\mathcal{S}p}

\theoremstyle{definition}
\newtheorem{de}{Definition}[section]
\newtheorem{nota}[de]{Notation}
\newtheorem{hypo}[de]{Hypothesis}

\theoremstyle{plain}
\newtheorem{thm}[de]{Theorem}
\newtheorem{lemma}[de]{Lemma}
\newtheorem{pro}[de]{Proposition}
\newtheorem{corr}[de]{Corollary}

\newtheorem*{thm*}{Theorem}
\newtheorem*{lemma*}{Lemma}
\newtheorem*{pro*}{Proposition}
\newtheorem*{corr*}{Corollary}

\theoremstyle{remark}
\newtheorem{rk}[de]{Remark}
\newtheorem{ex}[de]{Example}

\title{Subalgebras of the $\Z/2$-equivariant Steenrod algebra}
\author{Nicolas Ricka}
\email{ricka@math.univ-paris13.fr}
\address{Institut de Recherche Mathématique Avancée \\
 UMR 7501 \\
 7 rue René-Descartes \\
 67084 Strasbourg, \\
 France}

\begin{document}

\begin{abstract}
The aim of this paper is to study sub-algebras of the $\Z/2$-equivariant Steenrod algebra (for cohomology with coefficients in the constant Mackey functor $\mf$) which come from quotient Hopf algebroids of the $\Z/2$-equivariant dual Steenrod algebra. In particular, we study the equivariant counterpart of profile functions, exhibit the equivariant analogues of the classical $\ste(n)$ and $\mathcal{E}(n)$ and show that the Steenrod algebra is free as a module over these.
\end{abstract}

\maketitle

\section{Introduction}

It is a truism to say that the classical modulo 2 Steenrod algebra and its dual are powerful tools to study homology in particular and homotopy theory in general. In \cite{AM74}, Adams and Margolis have shown that sub-Hopf algebras of the modulo $p$ Steenrod algebra, or dually quotient Hopf algebras of the dual modulo $p$ Steenrod algebra have a very particular form. This allows an explicit study of the structure of the Steenrod algebra via so called {\em profile functions}.
 
These facts have both a profound and meaningful consequences in stable homotopy theory and concrete amazing applications (see \cite{Mar83} for many beautiful results strongly related to the structure of the Steenrod algebra).

Using the determination of the structure of the $\Z/2$-equivariant dual Steenrod algebra as a $RO(\Z/2)$-graded Hopf algebroid by Hu and Kriz in \cite{HK01}, we show two kind of results:
\begin{enumerate}
\item one that parallels \cite{AM74}, characterizing families of sub-Hopf algebroids of the $\Z/2$-equivariant dual Steenrod algebra, showing how to adapt classical methods to overcome the difficulties inherent to the equivariant world
\item and answers to some natural questions, arising from the essential differences between the non-equivariant and $\Z/2$-equivariant stable cohomological operations with coefficients in $\F$ and $\mf$ respectively.
\end{enumerate}

Our goal is to study sub-objects of the $\Z/2$-equivariant Steenrod algebra via the quotients of its dual, so we start by showing how duality theory {\em à la Boardman} fits into the context of $\Z/2$-equivariant stable cohomology operations in Section \ref{sec:Operations}.  When considering such equivariant generalisations, one must be careful since, for an equivariant cohomology theory $E^{\star}$, freeness of stable $E^{\star}$-cohomological operations as a module over the coefficient ring $E^{\star}_{\Z/2}$ (the $RO(\Z/2)$-graded abelian group of $\Z/2$-equivariant maps $S^{\star} \rightarrow E$) is not sufficient to apply Boardman's theory. The main result is Proposition \ref{pro_dualite_operationcoop}, which states that the appropriate notion in $\Z/2$-equivariant stable homotopy theory is freeness as a $RO(\Z/2)$-graded Mackey functor for $\Z/2$.

 We then show how this theory applies to the $\Z/2$-equivariant dual Steenrod algebra in Section \ref{sec:Steenrod}. The essential Proposition is \ref{corr_hypolibertehfd}, which asserts that the appropriate freeness condition is satisfied.

These results motivate our study of quotient Hopf algebroids of the $\Z/2$-equivariant Steenrod algebra. We develop in Section \ref{sec:Quotient} the required theoretical tools. After dealing with quotients of $RO(\Z/2)$-graded Hopf algebroids in general, we deal with the  cofreeness properties of various quotients of a given $RO(\Z/2)$-graded Hopf algebroid  in Theorem \ref{resultat_coliberte}.
The following section is devoted to the application of these ideas to the particular case of the $\Z/2$-equivariant dual Steenrod algebra. We define the appropriate notion of profile function in the equivariant case in Definition \ref{defhk} and deal with three questions:
\begin{enumerate}
\item Is there uniqueness of a pair of profile functions providing the same algebra quotient, and if yes, is there a preferred choice of profile functions?
\item Under what conditions on a pair of profile functions is the associated quotient algebra a quotient Hopf algebroid?
\item Under what additional conditions is this quotient Hopf algebroid free over $(H\mf_{\star})_{\Z/2}$?
\end{enumerate}

These questions leads respectively to the notions of {\em minimal pair}, the well-known notion of Hopf ideal, and {\em free pair}.

We end the paper by exhibiting in Section \ref{sec:Examples} some particular sub Hopf algebroids corresponding to the non-equivariant $\ste(n)$ and $\mathcal{E}(n)$, and exhibit some properties of these algebras. The principal result is Theorem \ref{thm_cofree}, which we apply to the particular cases $\ste^{\star}(n)$ and $\mathcal{E}^{\star}(n)$ to get the following corollary. Note that by construction, there are morphisms of Hopf algebroids
$$ \ste_{\star}(n) \rightarrow \ste_{\star}(n-1),$$
$$ \mathcal{E}_{\star}(n) \rightarrow \mathcal{E}_{\star}(n-1),$$
and
$$ \ste_{\star}(n) \rightarrow \mathcal{E}_{\star}$$
which gives to $ \ste_{\star}(n)$ a $\ste_{\star}(n-1)$-comodule structure, and accordingly for the other morphisms.

\begin{corr*}[\emph Corollary { \ref{cor_free}}]
For all $n \geq 0$, the Hopf algebroid $( H \mf_{ \star}, \ste_{ \star})$ is cofree as a comodule over its quotients
$ \mathcal{E}(n)_{ \star}$ and $ \ste_{ \star}(n)$, and these quotient Hopf algebroids are cofree over one other (whenever it makes sense).
\end{corr*}

Using duality theory for modules{\em à la Boardman}, this gives the next principal result of this paper.

\begin{thm}
For all $n \geq 0$, the $\Z/2$-equivariant Steenrod algebra $H \mf^{ \star}_{\Z/2}H\F$ is free as a module over
$ \mathcal{E}(n)^{ \star}$ and $ \ste(n)^{ \star}$, and these algebras are free over one another.
\end{thm}

As an application, we compute the $\Z/2$-equivariant modulo $2$ cohomology of the spectrum representing $K$-theory with reality. This type of results is essential when it comes to compute the Adams spectral sequence converging to $K$-theory with reality, as it allows us to use a change of coefficients property in the extension groups appearing in the $E^2$-page of the spectral sequence. This was one of the motivations of this work. We refer the interested reader to the ext charts appearing at the end of \cite{BG10} for applications of the analogous non-equivariant result to the computation of some $ko$-cohomology groups.

\thanks{Acknowledgement: I want to thank my advisor Geoffrey Powell for both the stimulating mathematical discussions we had, and his careful reading of early expositions of the results presented in this paper.}

\tableofcontents

\section{Preliminaries} \label{sec:Prelim}

\subsection{Equivariant cohomology theories and $RO(\Z/2)$-gradings}

Recall that the real representation ring of $\Z/2$, $RO(\Z/2)$ is isomorphic to $\Z \oplus \Z\alpha$, where $1$ stands for the one dimensional trivial representation, and $\alpha$ for the sign representation.
Consider the category of $\Z/2$-equivariant spectra indexed over a complete universe $\Z/2\sh$. We consider its usual model category structure, so that its associated homotopy category is the $\Z/2$-equivariant stable homotopy category $\Z/2 \sh$. We refer to \cite{LMS,GM95} for the constructions and definitions. The two following properties of $\Z/2\Sp$ are of particular importance.

\begin{rk} \label{rk_sv}
\begin{enumerate}
\item We implicitely chose a model for the category of $\Z/2$-equivariant spectra that has good monoidal properties, such as orthogonal equivariant $\Z/2$-spectra.
\item The meaning of the indexation over a complete universe is that the smash product with all the spheres $S^V$ obtained as the one point compactifcation of the real representations $V$ of $\Z/2$ is inverted up to homotopy.
\end{enumerate}
\end{rk}

\begin{nota}
\begin{itemize}
\item The $\star$ superscript will always denote a $RO(\Z/2)$-graded object, whereas the $*$ denotes a $\Z$-graded one.

\item There are ring morphisms $$i : \Z \inj RO(\Z/2)$$ (subring of trivial virtual representations), $$dim : RO(\Z/2) \surj \Z$$  which associates to a representation $V$ its virtual dimension, and $$twist : RO(\Z/2) \rightarrow \Z$$ which is defined by $twist(k+l\alpha) = l$.
\end{itemize}
\end{nota}

\begin{de}
\begin{itemize}
\item Let $\mathcal{O}$ be the full subcategory of $\Z/2\sh$ whose objects are $S^0$ and $\Z/2_+$. A Mackey functor for $\Z/2$ is an additive functor $\mathcal{O} \rightarrow \Ab$.
\item For $X$ and $Y$ two $\Z/2$-spectra, the stable classes of maps from $X$ to $Y$, denoted $[X,Y]$ is naturally a $RO(\Z/2)$-graded Mackey functor
$$ \Z/2\sh^{op} \times \Z/2\sh \rightarrow \mathcal{M}^{RO(\Z/2)}$$
defined, for $V \in RO(\Z/2)$ by
$$ \xymatrix{ [X,\Sigma^V Y]^{\Z/2} \ar@/^/[d]^{ \rho } \\ 
[X,\Sigma^V Y]^e, \ar@/^/[u]^{\tau} } $$
where, because of the adjunction $$[\Z/2_+\wedge X, \Sigma^V Y]^{\Z/2} \cong [X,\Sigma^V Y]^e \cong [X, \Z/2_+\wedge \Sigma^V Y]^{\Z/2},$$ the restriction and transfer of these Mackey functors are induced by the projection $\Z/2_+ \rightarrow S^0$ and its Spanier-Whitehead dual $S^0 \rightarrow \Z/2_+$ respectively.
\end{itemize}
\end{de}

\begin{nota}
Let $M$ be a Mackey functor. If $\rho_M : M_{\Z/2} \rightarrow M_e$ denotes the restriction of $M$, $\tau_M : M_e \rightarrow M_{\Z/2} $ its transfer, and $\theta_M :  M_e \rightarrow M_e$ the action of $\Z/2$ on $M_e$, we represent $M$ by the following diagram:
$$\xymatrix{ M_{\Z/2}  \ar@/^/[d]^{ \rho_M } \\ 
M_e. \ar@/^/[u]^{\tau_M} } $$
Observe that, since $\theta_M = \rho_M\tau_M -1$ this diagram suffices to determine $M$.
 We denote
$$(-)_e : \mathcal{M} \rightarrow \Z[\Z/2]-mod $$
and
$$ (-)_{\Z/2} : \mathcal{M} \rightarrow \Z-mod $$
the evaluation functors associated to the two objects of the orbit category for $\Z/2$.
\end{nota}

In particular $\Z/2$-equivariant cohomology theories are functors $\Z/2\sh^{op} \rightarrow \mathcal{M}^{RO(\Z/2)}$ satisfying equivariant analogues of the Eilenberg-Steenrod axioms.

We now turn to some elementary definitions in the abelian category of $RO(\Z/2)$-graded objects.

\begin{nota}
Let $(\ab, \otimes, 1_{\ab})$ be a closed symmetric monoidal abelian category.
Denote $\ab^{RO(\Z/2)}$ the category of $RO(\Z/2)$-graded objects. For an element $x$ of a $RO(\Z/2)$-graded object $M^{\star}$, denote $|x|  \in RO(\Z/2)$ its degree.
\end{nota}

As the underlying abelian group of $RO(\Z/2)$ is in particular an abelian monoid, one can define a symmetric monoidal structure on $\ab^{RO(\Z/2)}$ in a very natural way.

\begin{de}
Let $ \otimes :\ab^{RO(\Z/2)} \times \ab^{RO(\Z/2)} \rightarrow \ab^{RO(\Z/2)}$ the functor
\[
 (x  \otimes y)_{ \star} = \bigoplus_{V \oplus W = \star} x_V \otimes y_W.
\] For two $RO(\Z/2)$-graded objects $M,N$, $ \tau : M \otimes N \rightarrow N \otimes M$ is the natural isomorphism determined by $ ( \tau : M^V \otimes N^W \rightarrow N^W \otimes M^V )= (-1)^{dim(V) dim(W)} \tau_{\ab}$, where $\tau_{\ab}$ denotes the twist isomorphism in the category $\ab$.
\end{de}

\begin{lemma}
The triple $( \otimes, \tau, 1_{\ab})$, with the evident coherence morphisms give $\ab^{RO(\Z/2)}$ the structure of a closed symmetric monoidal category, where $\Hom$ is given in degree $\star$ by morphisms which increase the degree by $ \star$.
\end{lemma}

\begin{proof}
The proof, which is formally the same as in the classical case, is left to the reader.
\end{proof}

\begin{nota} \label{de_algebresrogr}
\begin{enumerate}
 \item A (commutative, unitary)  $RO(\Z/2)$-graded ring in $\ab$ is a (commutative, unitary) monoid in the category $( \otimes, \tau, 1_{\ab})$,
 \item for a commutative ring $H$, $H-Mod$ (the category of $H$-modules) and $H-Bimod$ (the category of $H$-bimodules) are defined via the usual diagrams,
\item for $R$ a cocommutative comonoid in $( \otimes, \tau, 1_{\ab})$, the category of $R$-comodules is defined as usual.
\end{enumerate}
\end{nota}

\subsection{The $\Z/2$-spectrum $H\mf$}

\begin{de}
Denote $\underline{\pi}_{\star} : \Z/2\sh \rightarrow \mathcal{M}^{RO(\Z/2)}$ the functor $X \mapsto [S^{\star}, X]$.
\end{de}

Using the functor $\underline{\pi}_*$, given by restricting $\pi_{\star}$ to integer gradings, one can extend the definition of Postnikov tower to the $\Z/2$-equivariant setting. This provides the appropriate notion of ordinary cohomology theory.

\begin{pro} \label{pro:functorH}
The $\Z/2$-equivariant Postnikov tower defines a $t$-structure on the $\Z/2$-equivariant stable homotopy category whose heart is isomorphic to the category $\mathcal{M}$ of Mackey functors for the group $\Z/2$. In particular, one has a functor
$$H : \mathcal{M} \rightarrow \Z/2\sh$$
which sends short exact sequences of Mackey functors to distinguished triangles of $\Z/2$-equivariant spectra.
\end{pro}

\begin{proof}
This proposition summarizes the results of \cite[Proposition I.7.14]{LMS} and \cite[Theorem 1.13]{Le95} in the particular case of the group with two elements.
\end{proof}

\begin{de}
Denote by $\mz$ the Mackey functor
$$\xymatrix{ \Z \ar@/^/[d]^{ =} \\ 
\Z \ar@/^/[u]^{ 2} } $$
and $\mf$ the Mackey functor
 $$\xymatrix{ \F \ar@/^/[d]^{ =} \\ 
\F. \ar@/^/[u]^{0} } $$
\end{de}

\begin{rk}
The constant Mackey functors in general, and $\mf$, $\mz$ in particular, play a special role in this context. One reason is that equivariant Eilenberg-MacLane spectra with coefficients in constant Mackey functors are exactly the $0th$-slices in the sense of the slice filtration (see \cite[Proposition 4.47]{HHR}). The spectrum $H\mz$ is even more particular since it identifies with $P^0_0(S^0)$ by \cite[Corollary 4.51]{HHR}. This particular role of $H\mz$ was first observed by Dugger in \cite{Du03}, studying an Atiyah-Hirzebruch spectral sequence for $K \mathbb{R}$-theory, which was really the slice spectral sequence of $K \mathbb{R}$.
\end{rk}

\subsection{The coefficient ring for $H\mf$}

We now write down the structure of the coefficient ring for $H\mf$-cohomology. For the complete computation, see \cite[p.371]{HK01}.

Recall from Remark \ref{rk_sv} that for a real representation, $V$ of $\Z/2$, the object $S^V$ stands for the one point compactification of $V$ as a $\Z/2$-space.

\begin{pro} \label{pro_cohompoint}
 The $RO( \Z/2)$-graded Mackey functor $H \mf_{ \star}$ is represented in the following picture. \\ The symbol $ \bullet$ stands for the Mackey functor $$ \xymatrix{ \Z/2 \ar@/^/[d] \\ 
0, \ar@/^/[u] } $$ 
and $L$ stands for $$\xymatrix{ \Z/2 \ar@/^/[d]^{ 0} \\ 
\Z/2 \ar@/^/[u]^{=} } $$
finally, $L_-$ represents the Mackey functor which consists only of $\F$ concentrated in $(-)_e$.
A vertical line represents the product with the Euler class $a$, which is the class of the map $S^0 \inj S^{\alpha}$. This product induces one of the following Mackey functor maps:
\begin{itemize}
\item the identity of $ \bullet$, 
\item the unique non-trivial morphism $ L \rightarrow \bullet$,
\item the unique non-trivial morphism $ \bullet \inj \mf$.
\end{itemize}
 
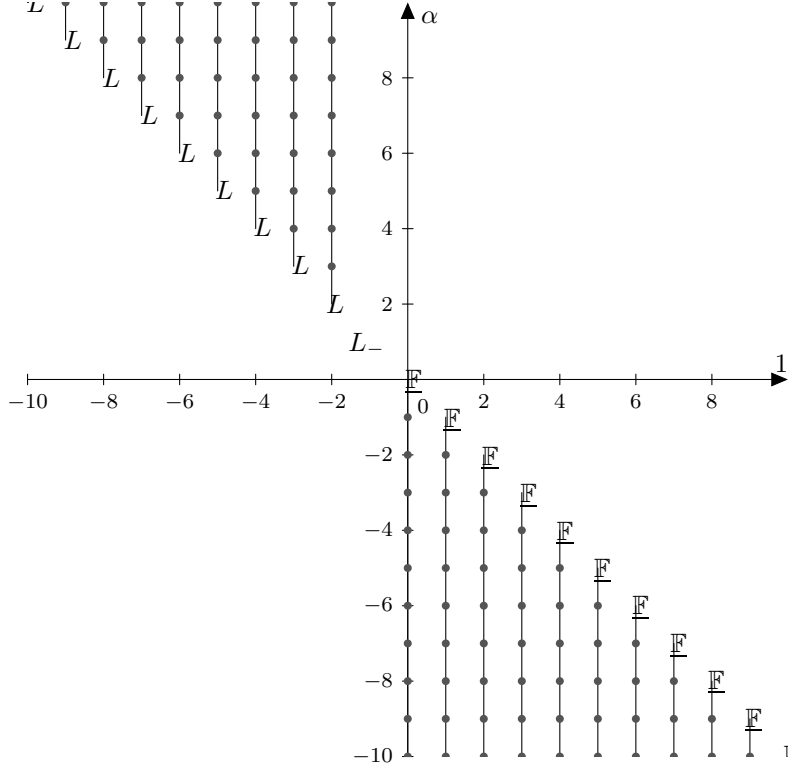
 \begin{center}
\definecolor{cqcqcq}{rgb}{0.75,0.75,0.75}
\definecolor{qqqqff}{rgb}{0.33,0.33,0.33}


\begin{tikzpicture}[line cap=round,line join=round,>=triangle 45,x=0.5cm,y=0.5cm]
\draw[->,color=black] (-10,0) -- (10,0);
\foreach \x in {-10,-8,-6,-4,-2,2,4,6,8}
\draw[shift={(\x,0)},color=black] (0pt,2pt) -- (0pt,-2pt) node[below] {\footnotesize $\x$};
\draw[->,color=black] (0,-10) -- (0,10);
\foreach \y in {-10,-8,-6,-4,-2,2,4,6,8}
\draw[shift={(0,\y)},color=black] (2pt,0pt) -- (-2pt,0pt) node[left] {\footnotesize $\y$};
\draw[color=black] (0pt,-10pt) node[right] {\footnotesize $0$};
\clip(-10,-10) rectangle (10,10);
\draw (-0.31,0.5) node[anchor=north west] {$\underline{ \mathbb{F}}$};
\draw (0.7,-0.51) node[anchor=north west] {$\underline{ \mathbb{F}}$};
\draw (1.7,-1.52) node[anchor=north west] {$\underline{ \mathbb{F}}$};
\draw (2.71,-2.52) node[anchor=north west] {$\underline{ \mathbb{F}}$};
\draw (3.67,-3.51) node[anchor=north west] {$\underline{ \mathbb{F}}$};
\draw (4.65,-4.51) node[anchor=north west] {$\underline{ \mathbb{F}}$};
\draw (5.66,-5.49) node[anchor=north west] {$\underline{ \mathbb{F}}$};
\draw (6.66,-6.5) node[anchor=north west] {$\underline{ \mathbb{F}}$};
\draw (7.65,-7.46) node[anchor=north west] {$\underline{ \mathbb{F}}$};
\draw (8.63,-8.47) node[anchor=north west] {$\underline{ \mathbb{F}}$};
\draw (9.64,-9.45) node[anchor=north west] {$\underline{ \mathbb{F}}$};
\draw (-1.81,1.5) node[anchor=north west] {$L_-$};
\draw (-12.25,12.46) node[anchor=north west] {$L$};
\draw (-11.24,11.45) node[anchor=north west] {$L$};
\draw (-10.28,10.49) node[anchor=north west] {$L$};
\draw (-9.3,9.51) node[anchor=north west] {$L$};
\draw (-8.29,8.5) node[anchor=north west] {$L$};
\draw (-7.29,7.5) node[anchor=north west] {$L$};
\draw (-6.3,6.51) node[anchor=north west] {$L$};
\draw (-5.32,5.51) node[anchor=north west] {$L$};
\draw (-4.31,4.53) node[anchor=north west] {$L$};
\draw (-3.31,3.52) node[anchor=north west] {$L$};
\draw (-2.4,2.49) node[anchor=north west] {$L$};
\draw (0,0) -- (0,-10);
\draw (1,-1) -- (1,-10);
\draw (2,-2) -- (2,-10);
\draw (3,-3) -- (3,-10);
\draw (4,-4) -- (4,-10);
\draw (5,-5) -- (5,-10);
\draw (6,-6) -- (6,-10);
\draw (7,-7) -- (7,-10);
\draw (8,-8) -- (8,-10);
\draw (9,-9) -- (9,-10);
\draw (10,-10) -- (10,-10);
\draw (-2,2) -- (-2,10);
\draw (-3,3) -- (-3,10);
\draw (-4,4) -- (-4,10);
\draw (-5,5) -- (-5,10);
\draw (-6,6) -- (-6,10);
\draw (-7,7) -- (-7,10);
\draw (-8,8) -- (-8,10);
\draw (-9,9) -- (-9,10);
\draw (-10,10) -- (-10,10);
\draw (-11,11) -- (-11,10);
\draw (9.4,0.9) node[anchor=north west] {$1$};
\draw (0.1,10) node[anchor=north west] {$\alpha$};
\begin{scriptsize}
\fill [color=qqqqff] (1,-2) circle (1.5pt);
\fill [color=qqqqff] (1,-3) circle (1.5pt);
\fill [color=qqqqff] (1,-4) circle (1.5pt);
\fill [color=qqqqff] (1,-5) circle (1.5pt);
\fill [color=qqqqff] (1,-6) circle (1.5pt);
\fill [color=qqqqff] (1,-7) circle (1.5pt);
\fill [color=qqqqff] (1,-8) circle (1.5pt);
\fill [color=qqqqff] (1,-9) circle (1.5pt);
\fill [color=qqqqff] (1,-10) circle (1.5pt);
\fill [color=qqqqff] (1,-11) circle (1.5pt);
\fill [color=qqqqff] (0,-1) circle (1.5pt);
\fill [color=qqqqff] (0,-2) circle (1.5pt);
\fill [color=qqqqff] (0,-3) circle (1.5pt);
\fill [color=qqqqff] (0,-4) circle (1.5pt);
\fill [color=qqqqff] (0,-5) circle (1.5pt);
\fill [color=qqqqff] (0,-6) circle (1.5pt);
\fill [color=qqqqff] (0,-7) circle (1.5pt);
\fill [color=qqqqff] (0,-8) circle (1.5pt);
\fill [color=qqqqff] (0,-9) circle (1.5pt);
\fill [color=qqqqff] (0,-10) circle (1.5pt);
\fill [color=qqqqff] (0,-11) circle (1.5pt);
\fill [color=qqqqff] (2,-3) circle (1.5pt);
\fill [color=qqqqff] (2,-4) circle (1.5pt);
\fill [color=qqqqff] (2,-5) circle (1.5pt);
\fill [color=qqqqff] (2,-6) circle (1.5pt);
\fill [color=qqqqff] (2,-7) circle (1.5pt);
\fill [color=qqqqff] (2,-8) circle (1.5pt);
\fill [color=qqqqff] (2,-9) circle (1.5pt);
\fill [color=qqqqff] (2,-10) circle (1.5pt);
\fill [color=qqqqff] (2,-11) circle (1.5pt);
\fill [color=qqqqff] (3,-4) circle (1.5pt);
\fill [color=qqqqff] (3,-5) circle (1.5pt);
\fill [color=qqqqff] (3,-6) circle (1.5pt);
\fill [color=qqqqff] (3,-7) circle (1.5pt);
\fill [color=qqqqff] (3,-8) circle (1.5pt);
\fill [color=qqqqff] (3,-9) circle (1.5pt);
\fill [color=qqqqff] (3,-10) circle (1.5pt);
\fill [color=qqqqff] (3,-11) circle (1.5pt);
\fill [color=qqqqff] (4,-5) circle (1.5pt);
\fill [color=qqqqff] (4,-6) circle (1.5pt);
\fill [color=qqqqff] (4,-7) circle (1.5pt);
\fill [color=qqqqff] (4,-8) circle (1.5pt);
\fill [color=qqqqff] (4,-9) circle (1.5pt);
\fill [color=qqqqff] (4,-10) circle (1.5pt);
\fill [color=qqqqff] (4,-11) circle (1.5pt);
\fill [color=qqqqff] (5,-6) circle (1.5pt);
\fill [color=qqqqff] (5,-7) circle (1.5pt);
\fill [color=qqqqff] (5,-8) circle (1.5pt);
\fill [color=qqqqff] (5,-9) circle (1.5pt);
\fill [color=qqqqff] (5,-10) circle (1.5pt);
\fill [color=qqqqff] (5,-11) circle (1.5pt);
\fill [color=qqqqff] (6,-7) circle (1.5pt);
\fill [color=qqqqff] (6,-8) circle (1.5pt);
\fill [color=qqqqff] (6,-9) circle (1.5pt);
\fill [color=qqqqff] (6,-10) circle (1.5pt);
\fill [color=qqqqff] (6,-11) circle (1.5pt);
\fill [color=qqqqff] (7,-8) circle (1.5pt);
\fill [color=qqqqff] (7,-9) circle (1.5pt);
\fill [color=qqqqff] (7,-10) circle (1.5pt);
\fill [color=qqqqff] (7,-11) circle (1.5pt);
\fill [color=qqqqff] (8,-9) circle (1.5pt);
\fill [color=qqqqff] (8,-10) circle (1.5pt);
\fill [color=qqqqff] (8,-11) circle (1.5pt);
\fill [color=qqqqff] (9,-10) circle (1.5pt);
\fill [color=qqqqff] (9,-11) circle (1.5pt);
\fill [color=qqqqff] (10,-11) circle (1.5pt);
\fill [color=qqqqff] (-2,11) circle (1.5pt);
\fill [color=qqqqff] (-2,10) circle (1.5pt);
\fill [color=qqqqff] (-2,9) circle (1.5pt);
\fill [color=qqqqff] (-2,8) circle (1.5pt);
\fill [color=qqqqff] (-2,7) circle (1.5pt);
\fill [color=qqqqff] (-2,6) circle (1.5pt);
\fill [color=qqqqff] (-2,5) circle (1.5pt);
\fill [color=qqqqff] (-2,4) circle (1.5pt);
\fill [color=qqqqff] (-2,3) circle (1.5pt);
\fill [color=qqqqff] (-3,11) circle (1.5pt);
\fill [color=qqqqff] (-3,10) circle (1.5pt);
\fill [color=qqqqff] (-3,9) circle (1.5pt);
\fill [color=qqqqff] (-3,8) circle (1.5pt);
\fill [color=qqqqff] (-3,7) circle (1.5pt);
\fill [color=qqqqff] (-3,6) circle (1.5pt);
\fill [color=qqqqff] (-3,5) circle (1.5pt);
\fill [color=qqqqff] (-3,4) circle (1.5pt);
\fill [color=qqqqff] (-4,11) circle (1.5pt);
\fill [color=qqqqff] (-4,10) circle (1.5pt);
\fill [color=qqqqff] (-4,9) circle (1.5pt);
\fill [color=qqqqff] (-4,8) circle (1.5pt);
\fill [color=qqqqff] (-4,7) circle (1.5pt);
\fill [color=qqqqff] (-4,6) circle (1.5pt);
\fill [color=qqqqff] (-4,5) circle (1.5pt);
\fill [color=qqqqff] (-5,11) circle (1.5pt);
\fill [color=qqqqff] (-5,10) circle (1.5pt);
\fill [color=qqqqff] (-5,9) circle (1.5pt);
\fill [color=qqqqff] (-5,8) circle (1.5pt);
\fill [color=qqqqff] (-5,7) circle (1.5pt);
\fill [color=qqqqff] (-5,6) circle (1.5pt);
\fill [color=qqqqff] (-6,11) circle (1.5pt);
\fill [color=qqqqff] (-6,10) circle (1.5pt);
\fill [color=qqqqff] (-6,9) circle (1.5pt);
\fill [color=qqqqff] (-6,8) circle (1.5pt);
\fill [color=qqqqff] (-6,7) circle (1.5pt);
\fill [color=qqqqff] (-7,11) circle (1.5pt);
\fill [color=qqqqff] (-7,10) circle (1.5pt);
\fill [color=qqqqff] (-7,9) circle (1.5pt);
\fill [color=qqqqff] (-7,8) circle (1.5pt);
\fill [color=qqqqff] (-8,11) circle (1.5pt);
\fill [color=qqqqff] (-8,10) circle (1.5pt);
\fill [color=qqqqff] (-8,9) circle (1.5pt);
\fill [color=qqqqff] (-9,11) circle (1.5pt);
\fill [color=qqqqff] (-9,10) circle (1.5pt);
\fill [color=qqqqff] (-10,11) circle (1.5pt);
\end{scriptsize}
\end{tikzpicture}

\end{center}

\end{pro}

\begin{nota} \label{nota:sigma}
We call $\sigma^{-1}$ the non trivial element in degree $(1-\alpha)$, so that $(H\mf)_{\Z/2}$ contains $\F[a, \sigma^{-1}]$ as a subalgebra.
\end{nota}

We finish this subsection by a lemma about Mackey functors, relating the $ \Z[a]/(2a)$-module structure on $([-,-]^{\star})_{\Z/2}$ with the $RO( \Z/2)$-graded Mackey functor structure of $ [-,-]^{\star}$.

\begin{lemma} \label{lemma_mackey_a}
 Let $E$ be a $ \Z/2$-spectrum.
\begin{enumerate}
 \item $Im(a)= Ker( \rho)$ where $ \rho : (E_{ \star})_{\Z/2} \rightarrow  (E_{ \star})_e $ stands for the restriction of the Mackey functor $E_{ \star}$.
\item $ Ker(a) = Im( \tau)$ where $ \tau $ is the transfer.
\end{enumerate}
\end{lemma}

\begin{proof}
 These are consequences of the existence of the long exact sequence associated to the distinguished triangle $$ \Z/2_+ \rightarrow S^0 \rightarrow S^{ \alpha}$$
in the stable $\Z/2$-equivariant category.
\begin{enumerate}
 \item Apply the exact functor $ [ -, \Sigma^{ - \star} E]^{\Z/2}$ to the triangle. We have:
$$ \xymatrix{  [S^{ \alpha}, \Sigma^{- \star} E]^{\Z/2} \ar[r] \ar@{ = }[d] & [S^0, \Sigma^{- \star} E]^{\Z/2} \ar[r] \ar@{ = }[d] & [ \Z/2_+, \Sigma^{- \star} E]^{\Z/2} \ar@{ = }[d] \\
\underline{ \pi}_{ \star+ \alpha}(E)_{\Z/2} \ar[r]^a & \underline{ \pi}_{ \star}(E)_{\Z/2} \ar[r]^{ \rho} & \underline{ \pi}_{ \star}(E)_e \\ }$$
where the rows are exact. The first point follows.
\item Apply the exact functor $ [ S^{ \star}, (-) \wedge E]_{e}$ to the triangle. We have:
$$ \xymatrix{  [S^{ \star}, \Z/2_+ \wedge E]^{\Z/2} \ar[r] \ar@{ = }[d] & [S^{ \star}, E]^{\Z/2} \ar[r] \ar@{ = }[d] & [ \Sigma^{ \star - \alpha}, E]^{\Z/2} \ar@{ = }[d] \\
\underline{ \pi}_{ \star}(E)_e \ar[r]^{ \tau} & \underline{ \pi}_{ \star}(E)_{\Z/2} \ar[r]^{a} & \underline{ \pi}_{ \star}(E)_{\Z/2} \\ }$$
where the rows are exact. The second point follows.
\end{enumerate}
\end{proof}

\section{Operations in $RO(\Z/2)$-graded cohomology} \label{sec:Operations}

\subsection{$RO(\Z/2)$-graded Hopf algebroids}

The role of this section is to set up the appropriate structure we need to deal with homology cooperations. These constructions are the $RO(\Z/2)$-graded analogue of the standard notions of Hopf algebroid, and the machinery considered by Boardman \cite{Bo95}.
We now define two monoidal structures on the category of $H$-bimodules.

\begin{de}
Let $H$ be a commutative $RO(\Z/2)$-graded ring.
\begin{enumerate}
\item Let $ \underset{(r,l)}{ \otimes} : H-Bimod \times H-Bimod \rightarrow H-Bimod$ be the tensor product over $H$, with respect to the right $H$-module structure on the first argument and the left $H$-module structure on the second. The unit for this structure is $H$ with the $H$-bimodule structure induced by the product on both sides.
\item Let $ \underset{(l,l)}{ \otimes} : H-Bimod \times H-Bimod \rightarrow H-Bimod$ be the tensor product over $H$, with respect to the left $H$-module structure on both arguments. The $H$-bimodule structure on $M \underset{(l,l)}{ \otimes} N$ is given by $h( m \otimes n) h' := (mh) \otimes (nh')$ for $h,h' \in H$ and $n\in N$, $M \in M$. Again, $H$ is the unit for this monoidal structure.
\item When $A$ is an $H$-bimodule and $M$ is a right $H$-module (resp. a left $H$-module), we can still define $M \underset{(r,l)}{\otimes} A$ (resp. $M \underset{(l,l)}{\otimes} A$, $A \underset{(r,l)}{\otimes} M$), as a right (resp. right, left) $H$-module.
\end{enumerate}

\end{de}

\begin{nota} \label{de_algebroidetc}
Let $H$ be a commutative $RO(\Z/2)$-graded ring.
\begin{enumerate}
\item $H-Alg$ is the category of monoids in $(H-mod, \otimes, H)$ and monoid morphisms,
\item $Alg$ is the category of $RO(\Z/2)$-graded algebra, that is $ \Z-Alg$ where $ \Z$ is concentrated in degree $0 \in RO( \Z/2)$. 
\item A $RO(\Z/2)$-graded Hopf algebroid $(H, A)$ is a cogroupoid object in $\Ab^{RO(\Z/2)}$.
\end{enumerate}
\end{nota}

Equivalently, a $RO(\Z/2)$-graded Hopf algebroid structure consists in the following data.

\begin{de}[~\emph{cf \cite[Definition A1.1.1]{Ra86}}]  \label{de_equivalgebroidetc}
Let $(H,A)$ be a pair of $RO(\Z/2)$-graded algebras. Denote by $ \mu$ the product of the algebra $A$.
A $RO( \Z/2)$-graded Hopf algebroid structure on $(H,A)$ consists on the following data
\begin{enumerate}
\item a left unit $ \eta_L : H \rightarrow A$,
\item a right unit $ \eta_R : H \rightarrow A$,
\item a coproduct $ \Delta : A \rightarrow A  \underset{(r,l)}{ \otimes} A$,
\item a counit $ \epsilon : A \rightarrow H$,
\item an antipode $ c : A \rightarrow A$,
\end{enumerate}
satisfying
\begin{enumerate}
\item $A$ has a $H$-bimodule structure, induced by $ \eta_L$ (for the left module structure) and  $ \eta_R$ (for the right module structure), and $ \Delta$ and $ \epsilon$ are $H$-bimodule morphisms,
\item the equalities \begin{itemize}
                    \item $ \epsilon \eta_L = \epsilon \eta_R = Id_{H}$
		    \item $( Id_{A} \underset{(r,l)}{ \otimes} \epsilon) \circ \Delta = ( \epsilon \underset{(r,l)}{ \otimes} Id_{A} ) \circ \Delta = Id_{A}$
   		    \item $( Id_{A} \underset{(r,l)}{ \otimes} \Delta) \circ \Delta = ( \Delta \underset{(r,l)}{ \otimes} Id_{A} ) \circ \Delta$
		    \item $c \eta_L = \eta_R$
		    \item $c \eta_R = \eta_L$
		    \item $c \circ c = Id_{A}$
                   \end{itemize}
are satisfied.
\item and there exists dotted arrows making the following diagram commute:
$$ \xymatrix@C=2.5cm{ A & \ar[l]_{ \mu \circ(c \otimes Id_A)} A \otimes_{ \mathbb{Z}} A \ar[d] \ar[r]^{ \mu \circ( Id_A \otimes c)} & A \\
& A  \underset{(r,l)}{ \otimes} A \ar@{-->}[ul]  \ar@{-->}[ur]  & \\
H \ar[uu]^{ \eta_R} & A \ar[l]_{ \epsilon} \ar[r]^{ \epsilon} \ar[u]^{ \Delta} & H \ar[uu]_{ \eta_L}} $$
\end{enumerate}
\end{de}

\begin{proof}[References for the proof:]
See \cite[Definition A1.1.1]{Ra86} and subsequent discussion (p.302) for this property in the category of $ \mathbb{N}$-graded  algebras over a ring $K$. The proof is, {\em mutatis mutandis}, the same in the $RO(\Z/2)$-graded context.
\end{proof}

We have all the usual notions attached to Hopf algebroids in the $RO(\Z/2)$-graded context, and in particular:

\begin{de} \label{de_primitifalghopf}
Let $(H,A)$ be a $RO(\Z/2)$-graded Hopf algebroid, we say that $x \in A$ is primitive if $ \Delta(x) = 1 \otimes_H x + x \otimes_H 1$.
\end{de}

\subsection{Duality between modules and comodules {\em à la Boardman}}  \label{subsection_otimesroggrad}

Following Boardman \cite{Bo95}, we now study a form of duality between modules and comodules which mimics the duality between cohomological operations and homological cooperations in the non-equivariant context.

\begin{de}[\emph{\cite[definitions appearing in \S 10 and Definition 11.11]{Bo95}}] \label{de_modulecomoduleboard}
 $\ $ \\
\begin{enumerate}
\item Let $A^{ \star}$ be an $H$-bimodule and a $RO(\Z/2)$-graded ring. An $A^{ \star}$-module {\em à la Boardman} is an $RO(\Z/2)$-graded filtered complete Hausdorff $H$-module $M$, together with a continuous $H$-module morphism $$ \lambda : A^{ \star} \underset{(r,l)}{ \otimes} M \rightarrow M$$ with the usual commutative diagrams defining a module over a ring.
\item Let $A_{ \star}$ be an $RO( \Z/2)$-graded Hopf algebroid. An $A_{ \star}$-comodule {\em à la Boardman}  is an $RO(\Z/2)$-graded filtered complete Hausdorff $H$-module $M$, together with a continuous $H$-module morphism $$ \rho : M \rightarrow M \hat{ \underset{(l,l)}{ \otimes}} A_{ \star},$$ where the action of $H$ on $M \hat{ \underset{(l,l)}{ \otimes}} A_{ \star}$ is given by $h ( m \otimes s) = m \otimes \eta_R(h)s$, for $m \otimes s \in M \hat{ \underset{(l,l)}{ \otimes}} A_{ \star}$ and $h \in H$.
\end{enumerate}
\end{de}

\begin{rk}
The difference between these notions and the more naive notions of module and comodule (see Definition \ref{de_algebresrogr}) is the completion of the tensor product ( $M \hat{ \underset{(l,l)}{ \otimes}} A_{ \star}$).
\end{rk}

The importance of these notions for us appears in the following extension to the $RO(\Z/2)$-graded context of a result of Boardman \cite{Bo95}.

\begin{pro}[\emph{\cite[Theorem 11.13]{Bo95}}] \label{pro_thmboardman}
Let $A_{ \star}$ be a $RO( \Z/2)$-graded Hopf algebroid, and suppose that $A_{ \star}$ is a free left $H$-module. Denote by $A^{ \star} = \Hom_{H-Mod}(A_{ \star}, H)$. Then, there is an equivalence of categories between the category of $A^{ \star}$-modules and the category of $A_{ \star}$-comodules.
\end{pro}

\subsection{Hopf algebroid structure on cooperations in homology with respect to a flat commutative ring spectrum}

\begin{de}
Let $E$ be a commutative ring $ \Z/2$-spectrum. We say that $E$ is flat if the $E_{\star}^{\Z/2}$-module $E_{\star}^{\Z/2}E$ is flat.
\end{de}

\begin{pro}
Let $E$ be a flat commutative ring $\Z/2$-spectrum. Then the unit $\eta : S^0 \rightarrow E$ and multiplication $\mu : E \wedge E \rightarrow E$ induce an $RO(\Z/2)$-graded Hopf algebroid structure on the pair $(E_{\star}^{\Z/2}, E_{\star}^{\Z/2}E)$.
Moreover, this Hopf algebroid is flat.

For any $\Z/2$-spectrum $X$, $E^{\star}_{\Z/2}(X)$ with the topology induced by the skeletal filtration is a comodule {\em à la Boardman} over this $RO(\Z/2)$-graded Hopf algebroid.
\end{pro}

\begin{proof}
The constructions of \cite[ Chapter 12 and 13]{Ad74}, \cite[Section 2.2]{Ra86} extend to the $RO(\Z/2)$-graded setting. The comodule structure is given by a $RO(\Z/2)$-graded analogue of \cite[Theorem 4.2]{Bo95}.
\end{proof}

Thus, Theorem \ref{de_modulecomoduleboard} holds in this setting, and provides an equivalence between the category of $E_{\star}^{\Z/2}E$-comodules {\em à la Boardman} and the category of $Hom_{E_{\star}^{\Z/2}}( E_{\star}^{\Z/2}E, E_{\star}^{\Z/2})$-modules. We want to compare this $Hom_{E_{\star}^{\Z/2}}( E_{\star}^{\Z/2}E, E_{\star}^{\Z/2})$-module structure to the more familiar $E^{\star}_{\Z/2}E$-module structure on $E^{\star}$-cohomology.

In general, the relationship between stable cohomological operations, given by $[E,E]^{ \star}_{\Z/2}$ and cooperations, given by $(E_{\star}^{\Z/2}, E_{\star}^{\Z/2}E)$ is delicate. In order to obtain an isomorphism 

$$Hom_{E_{\star}^{\Z/2}}( E_{\star}^{\Z/2}E, E_{\star}^{\Z/2}) \cong E^{ \star}_{\Z/2}E$$

 we need the following additional hypothesis on $E$ :
 \begin{hypo} \label{hypothese_libre} The $RO(\Z/2)$-graded Mackey functor $E_{ \star} E$ is isomorphic to a direct sum $ \bigoplus_{i \in I} \Sigma^{V_i} E_{ \star}$ as a $ E_{ \star}$-module for some $V_i \in RO(\Z/2)$.   \end{hypo}

Generally, for a commutative ring $\Z/2$-spectrum $E$, freeness of an $E$-module can be understood in purely algebraic terms.

\begin{pro} \label{pro_liberteemod}
 Let $X$ be a $E$-module such that $\underline{\pi}_{ \star}X$ is isomorphic (as an $RO(\Z/2)$-graded Mackey functor) to $ \bigoplus_{i \in I} \Sigma^{V_i} E_{ \star}$, for some $V_i \in RO( \Z/2)$.
Then, there exists an $E$-module weak equivalence $ \bigvee_{i \in I} \Sigma^{V_i} E \simeq X $.
\end{pro}

\begin{proof}
Let $ \{b_i \}_{i \in I}$ be a basis of $E_{ \star}X$ as a $E_{ \star}$-module, with $ | b_i| = V_i$.
Consider $ \bigvee_{i \in I} S^{V_i} \stackrel{ \vee b_i}{ \rightarrow} X$. This provides a $E$-module morphism
$$ \bigvee_{i \in I} \Sigma^{ deg(b_i)}E \rightarrow X.$$
This $ \Z/2$-spectrum map induces an isomorphism in homotopy Mackey functors, and thus is a weak equivalence.
\end{proof}

\begin{pro} \label{pro_dualite_operationcoop}
Suppose that $E$ satisfies Hypothesis \ref{hypothese_libre}. Then, there is a $RO(\Z/2)$-graded ring and $E_{ \star}^{\Z/2}$-module isomorphism:
\begin{align} Hom_{E_{\star}^{\Z/2}}( E_{\star}^{\Z/2}E, E_{\star}^{\Z/2}) \cong E^{ \star}_{\Z/2}E \label{dualiteoperationcooperation}  \end{align} 
where the ring structures are
\begin{enumerate}
\item induced by composition of morphisms $E \rightarrow E$ for $E^{ \star}_{\Z/2}E = \pi_{ \star}( End(E))_{\Z/2}$
\item given by the composite $ \Hom_{E_{\star}^{\Z/2}}(E_{\star}^{\Z/2}E,E_{\star}^{\Z/2}) \otimes_{E_{ \star}^{\Z/2}}  \Hom_{E_{\star}^{\Z/2}}(E_{\star}^{\Z/2}E,E_{\star}^{\Z/2}) \rightarrow  \Hom_{E_{\star}^{\Z/2}}(E_{\star}^{\Z/2}E \otimes_{E_{ \star}^{\Z/2}} (E_{\star}^{\Z/2}E, E_{\star}^{\Z/2}) \stackrel{ \Delta^*}{ \rightarrow} \Hom_{E_{\star}^{\Z/2}}(E_{\star}^{\Z/2}E, E_{\star}^{\Z/2}).$
\end{enumerate}
\end{pro}

\begin{proof}[\emph{ \textit{Proof} (indications, see \cite[proof of Theorem 9.25]{Bo95})}]
As $E$ satisfies Hypothesis \ref{hypothese_libre}, choose a basis $ \{b_i \}_{i \in I}$ of the $E_{ \star}^{\Z/2}$-module $E_{ \star}^{\Z/2}E$, so that there is a weak equivalence $E \wedge E \simeq \bigvee_i \Sigma^{ deg(b_i)}E$. Then, the map

\begin{eqnarray*}
\pi^{-1} : \Hom_{E_{ \star}^{\Z/2}}(E_{ \star}^{\Z/2}E, E_{ \star}^{\Z/2}) & \rightarrow &  \pi_{ \star} \Hom_{E-Mod}(E \wedge E, E) \\ & \cong & \pi_{ \star}( Hom(E,E)) = E^{ \star}_{\Z/2}E
\end{eqnarray*}

defined for $f : b_i \mapsto e_i \in E_{ \star}$ by the formula
 $$ \pi^{-1}(f) : \xymatrix{  \bigvee_i \Sigma^{ deg(b_i)}E \ar[r]^{ \quad \ \vee e_i \wedge E} &  E \wedge E \ar[r]^{ \quad \mu} & E }$$ 
provides an inverse isomorphism to the map $$ \pi_{ \star} : E^{ \star}_{\Z/2}E \cong \pi_{ \star} \Hom_{E-Mod}(E \wedge E, E) \rightarrow \Hom_{E_{ \star}^{\Z/2}}(E_{ \star}^{\Z/2}E, E_{ \star}^{\Z/2}).$$ 

Commutativity of the diagram \cite[Diagram 4.18]{Bo95} implies that this isomorphism is a ring isomorphism.
\end{proof}

\begin{corr} \label{corr_primindec}
Let $E$ be a ring $ \Z/2$-spectrum satisfying Hypothesis \ref{hypothese_libre}. Then, by the isomorphism provided by Proposition \ref{pro_dualite_operationcoop}, a stable $E^{\star}$-cohomology operation $x \in E^{ \star}_{\Z/2}E$ is indecomposable if and only if $x^{ \vee} \in E_{ \star}^{\Z/2}E$ is primitive.
\end{corr}

\section{The $\Z/2$-equivariant Steenrod algebra and duality} \label{sec:Steenrod}

We now turn to the study of the central object of this paper: the $\Z/2$-equivariant Steenrod algebra.

\begin{de}
Let $\ste^{\star} = H\mf^{\star}_{\Z/2}H\mf$ be the algebra of stable $H\mf$-cohomology operations.
\end{de}

Recall that Hu and Kriz  \cite{HK01} computed a presentation of the $\Z/2$-equivariant modulo $2$ dual Steenrod algebra
 $$\ste_{\star} := H\mf_{\star}^{\Z/2}H\mf = H\mf_{\star}^{\Z/2}[\xi_{i+1},\tau_i|i\geq0]/I,$$
 where $I$ is the ideal generated by the relation $\tau_i^2 = a\xi_{i+1}+ (a\tau_0+\sigma^{-1})\tau_{i+1}$, and $\sigma^{-1}$ is the class defined in Notation \ref{nota:sigma}. From this, we deduce the following result.

\begin{pro} \label{pro_hypolibrepourhfd}
 The  $ H \mf_{ \star}^{\Z/2}$-module $H \mf_{ \star}^{\Z/2} H \mf$ is free over $$ \mathcal{B}_m := \{ \Pi_{i,j} \tau_i^{ \epsilon_i} \xi_{j}^{n(j)}, n(j) \in \mathbb{N}, \epsilon(i) \in \{ 0,1 \} \}.$$ We call $\mathcal{B}_m$ the \textit{monomial basis} of $H \mf_{ \star}^{\Z/2} H \mf$.
\end{pro}

\begin{proof}
The proof is quite straightforward, but we include it here for completeness.

We show that the $H \mf_{ \star}^{\Z/2}$-module morphism 
$$ \phi : H \mf_{ \star}^{\Z/2} \{ \mathcal{B}_m \} \rightarrow \ste_{ \star}$$
is an isomorphism. \\ 
Let $R$ be the ideal generated by $ a \tau_{k+1}+ \eta_R( \sigma^{-1}) \xi_{k+1} - \tau_k^2$ for $k \geq 0$ so that $ \ste_{ \star} \cong H \mf_{ \star}^{\Z/2}[ \xi_{i+1}, \tau_{i} | i \geq 0] / R$.
\begin{itemize}
\item { \bf Surjectivity: } surjectivity follows from the definition of $\mathcal{B}_m$. Consider an element $ \xi_1^{i_1} \hdots \xi_n^{i_n} \tau_0^{j_1} \hdots \tau_m^{j_m}$ in $\mathcal{B}_m$. For all $k$ such that $j_k \geq 2$, write  $$ \xi_1^{i_1} \hdots \xi_n^{i_n} \tau_0^{j_1} \hdots \tau_m^{j_m} \equiv \xi_1^{i_1} \hdots \xi_n^{i_n}( \Pi_{k|j_k \leq 1} \tau_{j_k} )( \Pi_{k|j_k \geq 2} \tau_{k}^{j_k-2}(a \tau_{k+1}+ \eta_R( \sigma^{-1}) \xi_{k+1})$$ modulo $R$. By induction over $max \{ j_k \}$, there is an element of $H \mf_{ \star}^{\Z/2} \{ \mathcal{B}_m \}$ whose image by $ \phi$ is $ \xi_1^{i_1} \hdots \xi_n^{i_n} \tau_0^{j_1} \hdots \tau_m^{j_m}$.
\item { \bf Injectivity: }  this is shown analogously to the non-equivariant odd case. First, see that $Ker( \phi) \cong H \mf_{ \star}^{\Z/2} \{ \mathcal{B}_m \} \cap R$. But for all $0 \neq r \in R$, $ \exists i_1, \hdots, i_n, j_1, \hdots, j_k$ and $ \exists k \geq k_0 \geq 0$ such that $j_k \geq 2$ and $$pr_{ H \mf^{ \star}  \xi_1^{i_1} \hdots \xi_n^{i_n} \tau_0^{j_1} \hdots \tau_m^{j_m}}(r) \neq 0.$$ By definition of $ \mathcal{B}_m$,  $ H \mf_{ \star}^{\Z/2} \{ \mathcal{B}_m \} \cap R = 0$.
\end{itemize}
\end{proof}

In order to apply Boardman's theorem for $\ste_{\star}$-comodules {\em à la Boardman}, we need the following result.

\begin{pro} \label{corr_hypolibertehfd}
There is an isomorphism of Mackey functors $$ H \mf_{ \star}(H \mf) \cong \bigoplus_{b \in \mathcal{B}_m} \Sigma^{ |b|} H \mf_{ \star}. $$
\end{pro}

\begin{proof}
We first show the result in degrees indexed over trivial virtual representations $ \star = * \in \mathbb{Z} \subset RO( \Z/2)$. Let $F =  \bigoplus_{b \in \mathcal{B}_m} \Sigma^{ |b|} H \mf_{ \star}^{\Z/2}$.
We construct an explicit Mackey functor isomorphism.
$$ \xymatrix{ H \mf_{ *}^{\Z/2}H \mf \ar[rr] \ar@/^/[d]^{ \rho} & & F  \ar@/^/[d]^{ \rho} \\ 
H \mf_{ \star}^e(H \mf) \ar[rr] \ar@/^/[u]^{ \tau} & & \bigoplus_{b \in \mathcal{B}_m} \Sigma^{ |b|} H \mf_{ \star}^e \ar@/^/[u]^{ \tau}. } $$
Proposition \ref{pro_hypolibrepourhfd} gives precisely the isomorphism $H \mf_{ *}^{\Z/2}H \mf \rightarrow F$ by restricting to integers degrees. \\

By definition of Eilenberg-MacLane spectra the underlying non-equivariant spectrum of $H\mf$ is $H\F$, so there is an isomorphism of $ \Z$-graded abelian groups: 
$$ H \mf_{ *}^e(H \mf) = \underline{\pi}_*( H \mf \wedge H \mf)_e = \ste_*.$$ 
Recall that for all $\Z/2$-spectrum $E$ whose underlying non-equivariant spectrum is $E^u$, $\underline{\pi}_{\star}(E)_{\Z/2} \cong \pi_{dim(\star)}(E^u) \cong \pi_*(E^u)[\sigma^{\pm 1}]$. Thus, we obtain a $RO( \Z/2)$-graded abelian groups isomorphism: 
$$ (H \mf_{ \star}^e(H \mf)) = \underline{\pi}_*( H \mf \wedge H \mf)_e[\sigma^{-1}] = \ste_*[\sigma^{-1}].$$ 
The result \cite[Theorem 6.41]{HK01} implies that, in integer grading, the product with the Euler class $a$ on $H \mf_{ *}H \mf$ is injective. By Lemma \ref{lemma_mackey_a}, we know that the transfer is trivial in these degrees. Thus, the trace is trivial too, and the $ \Z/2$-action on $(H \mf_{ \star}^e(H \mf))$ is trivial. \\

But we already have $ H \mf_{ \star}^e = \F[ \sigma^{ \pm 1}]$, so $$ \bigoplus_{b \in \mathcal{B}_m} \Sigma^{ |b|} H \mf_{ \star}^e = \bigoplus_{b \in \mathcal{B}_m} \Sigma^{ deg(b)} \F$$ with trivial $ \Z/2$ action. \\

We deduce that the $\Z$-graded algebra morphism
$$ \psi : (H \mf_{ *}(H \mf))_{ \Z/2} \rightarrow \bigoplus_{b \in \mathcal{B}_m} \Sigma^{ |b|} (H \mf_{ *})_{ \Z/2}$$ which sends, for all $i \geq 0$, the element $ \sigma^{ -2^i+1} \tau_i \in  (H\mf_{2^{i+1}-1}H\mf)_{ \Z/2}$ of the Steenrod algebra on $ \xi_{i+1} \in \ste^* = (H \mf_{ *}(H \mf))_{ \Z/2}$ is a $ \F[ \Z/2]$-module isomorphism. \\

Commutation with transfer is satisfied since these morphisms are trivial. \\

By Lemma \ref{lemma_mackey_a}, we know the coimage of the restriction morphism for $ H \mf_{ \star}(H \mf)$ and $ \bigoplus_{b \in \mathcal{B}_m} \Sigma^{ |b|} H \mf_{ \star}$. For dimension reasons (the vector spaces are of finite dimension in each degree), these two restriction morphisms are surjective. Thus, replacing  $  \psi^{-1}$ by the composition of $  \psi^{-1}$ with a $ \F$-vector space isomorphism if necessary, the morphism

$$ \xymatrix{ H \mf_{ *}^{\Z/2}H \mf \ar[rr]^{ \phi} \ar@/^/[d]^{ \rho} & & F  \ar@/^/[d]^{ \rho} \\ 
H \mf_{ \star}^e(H \mf) \ar[rr]^{ \psi} \ar@/^/[u]^{ \tau} & & \bigoplus_{b \in \mathcal{B}_m} \Sigma^{ |b|} H \mf_{ \star}^e \ar@/^/[u]^{ \tau}. } $$
is an isomorphism of Mackey functors.
\end{proof}

We finish this subsection with a comparison between the $\Z/2$-equivariant dual Steenrod algebra and its non-equivariant counterpart $\ste_*$, which use the Mackey functor structure we have just determined.

\begin{pro} \label{pro_comparaisonste}
Via the identification $H \mf_{ \star}^{e}(H \mf) \cong \ste_*[ \sigma^{ \pm 1}]$, the restriction of the Mackey functor $ H \mf_{ \star}(H \mf)$ yields an algebra map
$$ \rho : \ste_{ \star} \rightarrow \ste_*[ \sigma^{ \pm 1}].$$
Moreover it induces a $\Z$-graded Hopf algebroid morphism
$$ r :( (H \mf_{ \star})_e, \ste_{ \star}) \rightarrow ( \F, \ste_*)$$
with the $\Z$-graduation induced by $dim : RO(\Z/2) \rightarrow \Z$, and the identification $ \ste_*[ \sigma^{ \pm 1}]/( \sigma^{-1} - 1) \cong \ste_*$.
\end{pro}

\begin{proof}
By definition, the map $ \rho$ is induced by $ p_+ \wedge id_{H \mf \wedge H \mf} : \Z/2_+ \wedge H \mf \wedge H \mf \rightarrow H \mf \wedge H \mf$ where $p :  \Z/2\rightarrow *$ is the unique $ \Z/2$-set map. Now, $ \rho$ induces a Hopf algebroid map because the restriction is induced by a ring  $ \Z/2$-spectrum map $ p_+ \wedge id_{H \mf \wedge H \mf}$.
\end{proof}

 \begin{rk}
We choose to quotient out the ideal $ (\sigma^{-1}-1)$. Another possibility would have been to choose $( \sigma^{-1})$. The map would then have been the analogue for the dual Steenrod algebras of the map studied by Caruso \cite{Ca99}.
 \end{rk}

\section{Quotients of the dual Steenrod algebra} \label{sec:Quotient} \label{sec:QuotientII}

\subsection{Quotient Hopf algebroids}

\begin{de}[see A.1 \emph{\cite{Ra86}}]  \label{de_ideal}
Let $(H, A)$ be a $RO(\Z/2)$-graded Hopf algebroid. An ideal $I$ of $A$ is a $RO(\Z/2)$-graded Hopf algebroid ideal if
\begin{eqnarray*}
 \epsilon(I) & = & 0 \\
 \Delta(I) &  \subset &  I \otimes_H A \oplus A \otimes_H I \\
 c(I) & \subset & I  \\
\end{eqnarray*}
\end{de}

\begin{pro} \label{pro_quotienthopf}
Let $(H, A)$ be a $RO(\Z/2)$-graded Hopf algebroid and $I$ be a $RO(\Z/2)$-graded Hopf algebroid ideal of $(H, A)$.
Then there is a natural $RO(\Z/2)$-graded Hopf algebroid structure on $(H, A/I)$ such that the projection $A \surj A/I$ induces a $RO(\Z/2)$-graded Hopf algebroid morphism $$(H,A) \rightarrow (H,A/I).$$
\end{pro}

\begin{proof}
The proof is a $RO(\Z/2)$-graded version of the classical one.
\end{proof}

\begin{rk} \label{remarque_quotient_non_libre}
Even when $A$ is $H$-free, since the ring $H$ is not a field, $A/I$ is not necessarily a free $H$-module.
\end{rk}

\begin{de} \label{de_hopfconnexe}
A $RO(\Z/2)$-graded Hopf algebroid $(H,A)$ is connected if $A$ is generated as an $H$-module in degrees $V \in RO( \Z/2)$ such that $dim(V) \geq 0$ with one generator of degree zero, which is a copy of $H$ generated by $\eta_R$ and $\eta_L$.
\end{de}

\begin{ex}
The $RO(\Z/2)$-graded Hopf algebroid $(H \mf_{ \star}^{\Z/2}, \ste_{ \star})$ is connected.
\end{ex}

\begin{pro} \label{pro_conjconnexe}
 If $(H,A)$ is connected and $I$ is an ideal of  $A$ as an algebra satisfying
$$\epsilon(I)  =  0 $$ and $$\Delta(I)  \subset   I \otimes_H A \oplus A \otimes_H I.$$
Then, $I$ is a $RO(\Z/2)$-graded Hopf algebroid ideal.
\end{pro}

\begin{proof}
The proof is analogous to the one in the classical $\Z$-graded case, reasoning with $dim(|x|) \in \Z$ for the induction hypothesis.
\end{proof}

\subsection{A family of quotient algebras of {$\ste_{\star}$}}

In this subsection we define the particular quotient algebras of the equivariant Steenrod algebra we are interested in. The aim of the rest of this section is to develop a theory of profile functions for quotient Hopf algebroids of the $\Z/2$-equivariant dual Steenrod algebra.
The  results we prove here for the $\Z/2$-equivariant Steenrod algebra are analogous to those of Adams and Margolis \cite{AM74} . The notation used here is very similar to that of Adams and Margolis for the non-equivariant odd primary case.

\begin{de} \label{defhk}
We call {\em profile function} a pair of maps $(h,k)$, $$h : \mathbb{N} \backslash \{ 0 \} \rightarrow \mathbb{N} \cup \{ \infty \} $$  $$k: \mathbb{N}  \rightarrow \mathbb{N} \cup \{ \infty \}.$$
For profile functions $(h,k)$, denote by $I(h,k)$ the two-sided ideal of $ \ste_{ \star}$ generated by $ \xi_i^{2^{h(i)}}$  and $ \tau_i^{2^{k(i)}}$ (with the convention $x^{ \infty} = 0$).
\end{de}

We think of profile functions $(h,k)$ as a way to encode the quotient algebra $\ste_{\star}/I(h,k)$ of the $\Z/2$-equivariant Steenrod algebra. 

\begin{ex} \label{exemple_ideaux_egaux}
Denote $h = (0, \infty, \infty, \hdots )$.
We define two profile function which will serve as examples through the rest of this section: 
\begin{itemize}
 \item let $I_1 = I(h,k_1)$ for $k_1 = (1,0, \infty, \infty, \hdots )$,
\item let $I_2 = I(h,k_2)$ for  $k_2 = ( \infty,0, \infty, \infty, \hdots ).$
\end{itemize}
\end{ex}
Because of the relation $ \tau_0^2 = a \tau_1 + \eta_R(\sigma^{-1}) \xi_1$ in the $\Z/2$-equivariant dual Steenrod algebra, the two ideals $I_1$ and $I_2$ coincide. In the next subsection, for a given ideal $I$ of $\ste_{\star}$, we define a preferred choice of profile function defining $I$ called the minimal profile function.

\subsection{Minimality of profile functions}

\begin{de}
\begin{itemize}
 \item Define a partial order on profile functions $(h,k)$, by $(h,k) \leq (h',k')$ if $ \forall n \geq 0$, $h(n+1) \leq h'(n+1)$ and $k(n) \leq k'(n)$.
\item We say that a profile function $(h,k)$ is {\em minimal} if it is minimal among the profile functions $(h',k')$ such that $I(h',k')=I(h,k)$. 
\end{itemize}
\end{de}

\begin{lemma} \label{lemma_minequiv}
\begin{enumerate}
 \item A profile function $(h,k)$ is minimal if and only if
$ \forall i,n \geq 0$, $ \tau_i^{2^n} \in I(h,k)$ is equivalent to $n \geq k(i)$.
\item Let $(h,k)$ be a profile function.
Then the profile function $(h, \tilde{k})$ defined by
\begin{eqnarray*}
 \tilde{k} : \mathbb{N} \backslash \{ 0 \} &  \rightarrow & \mathbb{N} \cup \{ \infty \}  \\
n & \mapsto & Min \{ l \in \mathbb{N} | \tau_n^{2^l} \in I(h,k) \}.
\end{eqnarray*}
is the unique minimal profile function such that $I(h,k) = I(h, \tilde{k})$.
\end{enumerate}
\end{lemma}

\begin{proof}
\begin{enumerate}
 \item Let $(h,k)$ be a profile function. If $ \tau_i^{2^n} \not\in I(h,k)$, then $n \leq k(i)$, this proves $\Leftarrow$. \\

For $\Rightarrow$, let $(h,k)$ be a profile function. If $(h,k)$ do not satisfy the asserted hypothesis, there is an integer $i$ such that $ \tau_i^{2^{k(i)-1}} \in I(h,k)$. Let $ \tilde{k}(j)= \left\{ \begin{matrix} k(j) \text{ if } j \neq i \\ k(i)-1 \text{ if } i=j. \end{matrix} \right. $ Then, we have $I(h,k)= I(h, \tilde{k})$,  and $(h, \tilde{k}) <(h,k)$. Thus $(h,k)$ is not minimal.
\item First, observe that the two profile functions $(h,k)$ and $(h, \tilde{k})$ generates the same ideal. Moreover, by the first point, the profile function $(h, \tilde{k})$ is minimal. Thus, it remains to check that such a minimal couple is unique. It is a consequence of the first point: if $I(h,k) = I(h',k')$, then $h = h'$ and the first point implies that for any profile function $(h',k')$ generating the ideal $I(h,k)$, we have $(h, \tilde{k}) \leq (h',k')$.
\end{enumerate}
\end{proof}

\begin{ex} 
Consider the examples defined in Example \ref{exemple_ideaux_egaux}; the minimal profile function associated to $(h,k_2)$ is $(h,k_1)$. Indeed,
\begin{itemize}
\item the profile function $(h,k_2)$ is not minimal: $ \tau_0^2 = a \tau_1 + \eta_R( \sigma^{-1}) \xi_1 \in I(h,k)$ but $k_2(0) = \infty$, so the first point of Lemma \ref{lemma_minequiv} cannot be satisfied.
 \item the profile function $(h,k_1)$ is minimal.
\end{itemize}
\end{ex}

\subsection{Conditions on the generating relations to provide a quotient Hopf algebroid}

We now focus on necessary and sufficient conditions on a profile function $(h,k)$ so that $ \ste_{ \star}/I(h,k)$ has a Hopf algebroid structure induced by the quotient. \\

\begin{rk}
The following should be compared to the odd primary case of \cite{AM74}. The similarity between the two cases comes from the fact that the formulas for the coproduct in $ \ste^p_*$, for odd $p$ (with the notations of \cite{AM74}) and $ \ste_{ \star}$ (with our notations) are the same.
\end{rk}

\begin{pro} \label{proihk}
Let $(h,k)$ be a minimal pair of profile functions.
The ideal $I(h,k)$ is a $RO(\Z/2)$-graded Hopf algebroid ideal if and only if the profile function satisfy:

\begin{align}
 \forall i,j \geq 1, \  &  h(i) \leq j + h(i+j) \text{ or } h(j) \leq h(i+j) \label{eqn_ideal1}\\
  \forall i \geq 1, \  j \geq 0, \ &  h(i) \leq j+ k(i+j) \text{ or } k(j) \leq k(i+j) \label{eqn_ideal2}
\end{align}
\end{pro}

\begin{proof}
We check the equivalence between the three conditions of Definition \ref{de_ideal} and the equations \eqref{eqn_ideal1} and \eqref{eqn_ideal2}. \\

By definition of $I(h,k)$, $ \epsilon(I(h,k)) = 0$. Moreover, Proposition \ref{pro_conjconnexe} applies to the pair $( H \mf_{ \star}^{\Z/2}, \ste_{ \star}/I(h,k))$ giving that $c(I(h,k)) \subset I(h,k)$ is automatically satisfied if $ \Delta( I(h,k)) \subset I(h,k) \otimes \ste_{ \star} \oplus \ste_{ \star} \otimes I(h,k)$. \\

Thus, it suffices to show that $ \Delta( I(h,k)) \subset I(h,k) \otimes \ste_{ \star} \oplus \ste_{ \star} \otimes I(h,k)$ is equivalent to
\begin{align}
 \forall i,j \geq 1, &  h(i) \leq j + h(i+j) \text{ or } h(j) \leq h(i+j) \label{hypotheseh} \\
  \forall i \geq 1, j \geq 0, &  h(i) \leq j+ k(i+j) \text{ or } k(j) \leq k(i+j).
\end{align}

 We check that the conditions are equivalent, using the known structure of the $\Z/2$-equivariant dual Steenrod algebra.
The coproduct $ \Delta$ is an algebra morphism, so $I(h,k)$ satisfies $ \Delta( I(h,k)) \subset I(h,k) \otimes \ste_{ \star} \oplus \ste_{ \star} \otimes I(h,k)$ if and only if, for all $i \geq 1$, $ \Delta( \xi_i^{2^{h(i)}}) \in I(h,k) \otimes  \ste_{ \star}  \oplus  \ste_{ \star}  \otimes I(h,k)$ and for all $i \geq 0$, $ \Delta( \tau_i^{2^{k(i)}}) \in I(h,k) \otimes \ste_{ \star}  \oplus \ste_{ \star} \otimes I(h,k)$. But
$$ \Delta( \xi_n^{2^{h(n)}}) = \sum_{i+j=n} \xi_i^{2^{j+h(n)}} \otimes \xi_j^{2^{h(n)}}$$
and
$$ \Delta ( \tau_n^{2^{k(n)}}) = \sum_{i+j=n} \xi_i^{2^{j+k(n)}} \otimes \tau_j^{2^{k(n)}}.$$
By the first equation, $ \Delta( \xi_i^{2^{h(i)}}) \in I(h,k) \otimes  \ste_{ \star}  \oplus  \ste_{ \star}  \otimes I(h,k)$ is equivalent to equation \eqref{hypotheseh}.
For the second one, $ \tau_n^{2^{k(n)}} \in I(h,k)$ by definition of $I(h,k)$. Now, for all $i,j$ such that $i+j=n$, either $ \xi_i^{2^{j+k(n)}} \in I(h,k)$ or $ \tau_j^{2^{k(n)}} \in I(h,k)$. By minimality, the second condition is equivalent to $k(j) \leq k(n)$.
\end{proof}

\subsection{Freeness of quotients Hopf algebroids}

We now address the problem raised in Remark \ref{remarque_quotient_non_libre}. We will give a numerical condition on a pair of profile functions $(h,k)$ to ensure that the associated quotient algebra $\ste_{\star}/I(h,k)$ is free over $H\mf_{\star}^{\Z/2}$.

\begin{de} \label{de_libre}
We say that a pair of profile functions $(h,k)$ is {\em free} if it satisfies the following property:
$ \forall i \geq 0$, $ m \geq k(i)$, $j \leq m$,
$$k(i+m)= 0$$
and
$$ h(i+j) \leq m-j.$$
\end{de}

\begin{lemma} \label{lemma_condlibre}
Let $(h,k)$ be a free pair of profile functions. Then, for all $i \geq 0$ and $m \geq 0$, the following conditions are equivalent:
\begin{enumerate}
 \item $ \tau_i^{2^m} \in I(h,k)$,
\item in the decomposition of $ \tau_i^{2^m} $ in the monomial basis, i.e. $ \tau_i^{2^m}= \sum_j h_j x_j$, for some $h_j \in H \mf_{ \star}^{\Z/2}$ and $x_j \in \mathcal{B}_m$, we have that if $j \geq 0$, then $x_j \in I(h,k)$,
 \item $ \tau_{i+m} \in I(h,k)$ and $ \forall 0 \leq j \leq m-1$, $ \xi_{i+m-j}^{2^{j}} \in I(h,k)$.
\end{enumerate}
\end{lemma}

\begin{proof}
$2 \Leftrightarrow 3$: write $ \tau_i^{2^m}$ in the monomial basis. By the relation $ \tau_i^2 = a \tau_{i+1} + \eta_R( \sigma^{-1}) \xi_{i+1}$ in the $\Z/2$-equivariant dual Steenrod algebra, thus $ \tau_i^{2^m} = a^{2^{m-1}} \tau_{i+1}^{2^{m-1}} + \eta_R( \sigma^{-1})^{2^{m-1}} \xi_{i+1}^{2^{m-1}}$.
By induction on $m$, we find

\begin{equation} \label{eqntau2m}
 \tau_i^{2^m} = a^{2^m-1} \tau_{i+m} + \sum_{j=1}^m a^{2^m-2^j} \eta_R( \sigma^{-1})^{2^{j-1}} \xi_{i+m-j+1}^{2^j},
\end{equation}

the result follows.

$2 \Rightarrow 1 $: suppose $2$. Then equation (\eqref{eqntau2m}) implies that $ \tau_i^{2^m} \in I(h,k)$.

$ 1 \Rightarrow 2$ : if $ \tau_i^{2^m} \in I(h,k)$, suppose that the implication is true for all $m'<m$ and $i'<i$. Then, one of the following assertions is satisfied:
\begin{itemize}
\item $m \geq k(i)$, and thus freeness of $(h,k)$ implies that every element of the monomial basis appearing in equation \eqref{eqntau2m} are already in $I(h,k)$,
\item  $m < k(i)$. In that case, $ a^{2^{m-1}} \tau_{i+1}^{2^{m-1}} + \eta_R( \sigma^{-1})^{2^{m-1}} \xi_{i+1}^{2^{m-1}} \in I(h,k)$. The ideal $I(h,k)$ is generated by monomials, so $ \tau_{i+1}^{2^{m-1}} \in I(h,k)$ and $  \eta_R( \sigma^{-1})^{2^{m-1}} \xi_{i+1}^{2^{m-1}} \in I(h,k)$. The powers of the elements $ \tau_{i'}^{2^{m'}}$ appearing in this expression satisfies $m'<m$ and $i'<i$, so their decomposition in the monomial basis consists only in terms of $I(h,k)$ by induction hypothesis.
\end{itemize}
In each case, the induction step holds, concluding the proof.
\end{proof}

\begin{pro} \label{pro_quotientlibre}
If $(h,k)$ is a free pair of profile functions, then the $ H \mf_{ \star}^{\Z/2}$-module $ \ste_{ \star} / I(h,k)$ is free.
Moreover, a basis for this $H \mf_{ \star}^{\Z/2}$-module is
$$ \mathcal{B}_{(h,k)} = \{ [b_i] | b_i \in \mathcal{B}_m, b_i \not\in I(h,k) \} ,$$

or equivalently consists in elements of the monomial basis which are of the form $$ \Pi_{i \geq 0} \xi_{i+1}^{h_i} \tau_i^{ \epsilon_i}$$ where $ \forall i \geq 0$, $ \epsilon_i \leq 2^{k(i)}$ and $h_{i+1} \leq 2^{h(i+1)}$, $ \epsilon_i =0$ or $ 1$ and the product is finite.
\end{pro}

\begin{proof}
Consider the canonical $H \mf_{ \star}^{\Z/2}$-module morphism
$$  \phi : H \mf_{ \star}^{\Z/2} \{ \mathcal{B}_{(h,k)} \} \surj \ste_{ \star} / I(h,k).$$
Suppose that $(h,k)$ is free, we will show that $ \phi$ is a $H \mf_{ \star}^{\Z/2}$-module isomorphism.

Let $x \in Ker( \phi)$, {\em i.e.} $x = \sum h_ib_i$ for $h_i \in H \mf_{ \star}^{\Z/2}$ and $b_i \in \mathcal{B}_{(h,k)}$ such that $x \in I(h,k)$. Consequently, there exists  $x_i,y_i \in \ste_{ \star}$ such that
$$ x = \sum_{i \geq 0} x_i \xi_i^{2^{h(i)}} + y_i \tau_i^{2^{k(i)}}.$$

Thus for all $ i \geq 0$, 
\begin{itemize}
 \item each element of the monomial basis appearing in the decomposition of $x_i \xi_i^{2^{h(i)}}$ is an element of $I(h,k)$ because $ \xi_i^{2^{h(i)}} \in I(h,k)$,
 \item by Lemma \ref{lemma_condlibre}, each element of the monomial basis appearing in the decomposition of $ \tau_i^{2^{k(i)}}$ is in $I(h,k)$, so the same is true for element of the monomial basis appearing in the decomposition of $y_i \tau_i^{2^{k(i)}}$.
\end{itemize}

Consequently, we get  $x \in H \mf_{ \star}^{\Z/2} \{ \mathcal{B}_{(h,k)} \} \cap H \mf_{ \star}^{\Z/2} \{ \mathcal{B}_m \backslash \mathcal{B}_{(h,k)} \}$, thus $x = 0$.
\end{proof}

\section{Notable examples and properties} \label{sec:Examples}

\subsection{Some examples of quotient Hopf algebroids of $\ste_{\star}$}

\begin{de}
For a map $h : \mathbb{N} - \{0\} \rightarrow \mathbb{N} \cup \{\infty\}$, we denote by $J(h)$ the ideal of the modulo $2$ non-equivariant Steenrod algebra generated by the elements $\xi_i^{2^{h(i)}}$.
\end{de}

\begin{rk}
By Adams-Margolis \cite[Theorem 3.3]{AM74}, we have a necessary and sufficient condition on  $h$ such that $J(h)$ is a Hopf algebra ideal of the non-equivariant Steenrod algebra.
\end{rk}

\begin{pro} \label{pro_lienclassiqueequiv}
There is an injection of partially ordered sets
\begin{eqnarray*}
  \{ \text{quotient Hopf algebras of } \ste_*^2 \} & \rightarrow & \{  \text{quotient Hopf algebroids of } \ste_{ \star} \} \\
\ste_*^2/J(h) & \mapsto & \ste_{ \star}/I(h,0)
\end{eqnarray*}
where the partial order is induced by inclusion of ideals.
\end{pro}

\begin{proof}
Let $ \mathcal{B}_*$ be a Hopf algebra quotient of $ \ste^2_*$. By the main result of \cite{AM74}, we know that there exists a unique $h : \mathbb{N} \backslash \{ 0 \} \rightarrow \mathbb{N} \cup \{ \infty \}$ satisfying the conditions of Proposition \ref{proihk} for $k=0$ (the second one becomes trivial for $k=0$ and the first one becomes precisely the one asserted by \cite{AM74}) such that $ \mathcal{B}_*$ is isomorphic to
$ \ste^2_* / J(h)$. \\
Thus, Proposition \ref{proihk} implies that the application is well defined. The fact that it preserves the partial order is true by construction.
\end{proof}

The previous proposition allows us to define some examples of quotient Hopf algebroids and, by duality (Proposition \ref{pro_dualite_operationcoop} under Hypothesis \ref{hypothese_libre} which is satisfied by Proposition \ref{corr_hypolibertehfd}), this provides some particular sub-algebras of the $\Z/2$-equivariant Steenrod algebra.

\begin{de}
Define $ \widetilde{ \ste(n)_{ \star}} = \ste_{ \star} / I(h_n,0)$ for $h_n = (n-1,n-2,n-3, \hdots, 0,0 \hdots )$.
And denote $ \widetilde{ \ste(n)^{ \star}} = \Hom_{ H \mf_{ \star}^{\Z/2}}( \widetilde{ \ste(n)_{ \star}} , H \mf_{ \star}^{\Z/2})$ their dual algebras.
\end{de}

These algebras are analogous to the squares of the classical $\ste(n)^* \subset \ste^*_2$. We now define genuine analogues to the classical sub-Hopf algebras $\ste^*(n)$ and $\mathcal{E}^*(n)$ of $\ste^*_2$, which are respectively $\ste^{\star}(n)$ and $\mathcal{E}^{\star}(n)$.

\begin{de}
\begin{itemize}
\item Let $ \mathcal{E}(n)_{ \star} = \ste_{ \star} / I(0,k_n)$ for $k_n=(n,n-1,n-2, \hdots,0,0, \hdots)$.
Denote by $ \mathcal{E}( \infty)_{ \star}$ the quotient Hopf algebroid  $ \ste_{ \star} / I(0,k)$ for $k=( \infty, \infty, \infty, \hdots ) $,
and $ \mathcal{E}^{ \star}(n)$ and $ \mathcal{E}^{ \star}( \infty)$ their duals algebras.
\item Denote $  \ste(n)_{ \star} = \ste_{ \star} / I(h_n,k_n)$ for $h_n=(n-1,n-2,n-3, \hdots,0,0 \hdots)$ and $k_n=(n,n-1,n-2, \hdots,0,0, \hdots)$ and $  \ste(n)^{ \star}$ its dual algebra.
\end{itemize}
\end{de}

\begin{rk}
 Proposition \ref{pro_quotientlibre} and the fact that the profile functions defining the $ \mathcal{E}( \infty)_{ \star}$, $ \mathcal{E}(n)_{ \star}$, and $  \ste(n)_{ \star}$ are free implies that these are free $H \mf_{ \star}^{\Z/2}$-module, with basis $ \tau_0^{ \epsilon_0}, \hdots \tau_n^{ \epsilon_n}$, for some $ \epsilon_i \in \{0,1 \}$ ($n = \infty$ for $ \mathcal{E}( \infty)_{ \star}$).
\end{rk}

\subsection{Cofreeness $(H\mf_{\star}^{\Z/2},\ste_{\star})$ over its quotients}

We now answer the question of cofreeness of the various quotient Hopf algebroids of $(H\mf_{\star}^{\Z/2},\ste_{\star})$ over one another. 
It turns out to have a very general answer when modules are concentrated in positive {\em twists}.
We first define a version of the $\Z/2$-equivariant dual Steenrod algebra concentrated in positive {\em twists}.

\begin{de}
Let $( H_{ \geq}, \ste_{ \geq})$ be the Hopf algebroid $$( \F[a, \sigma^{-1} ], \F[a, \sigma^{-1} ][ \tau_i, \xi_i+1,i \geq 0]/ \tau_i^2 = a \tau_{i+1} + \eta_R( \sigma^{ -1}) \xi_{i+1})$$ with the same formulae as for the $\Z/2$-equivariant dual Steenrod algebra for the structure maps, that is:\begin{align*}
\Delta( \xi_n) = & \sum_{i=0}^n \xi_{n-i}^{p^i} \otimes \xi_i \\
\Delta( \tau_n) = & \sum_{i=0}^n \xi_{n-i}^{p^i} \otimes \tau_i + \tau_n \otimes 1
\end{align*}
\end{de}

We now recall the construction of the extension of objects for a small groupoid.

\begin{de}
 Let $ \mathcal{O}$ be the set of objects and $ \mathcal{M}$  the set of morphisms of a small groupoid. A a left $ (\mathcal{O}, \mathcal{M})$-module structure on a set $ \mathcal{Y}$ over $ \mathcal{O}$ is a map $$ \psi : \mathcal{M} \times_{ \mathcal{O} } \mathcal{Y} \rightarrow \mathcal{Y}$$ where $ \mathcal{M}$ is seen as a set over $\mathcal{O}$ via the source morphism which is compatible with projection, associative and unital.
\end{de}

\begin{pro} \label{pro_groupoidgeoffrey}
Let $ \mathcal{Y}$ be a left $( \mathcal{O}, \mathcal{M})$-module. Then there is a natural groupoid structure on $( \mathcal{Y}, \mathcal{M} \times_{ \mathcal{O}} \mathcal{Y})$ defined by:
\begin{enumerate}
\item identity: $ \mathcal{Y} = \mathcal{O} \times_{ \mathcal{O}} \mathcal{Y}  \stackrel{e \times \mathcal{Y}}{\longrightarrow}  \mathcal{M} \times_{ \mathcal{O}} \mathcal{Y}$ where $e$ is the identity,
\item source: $pr_{ \mathcal{Y}}$
\item target: $ \psi$
\item inverse: $c \ pr_{ \mathcal{M}} \times_{ \mathcal{O}} \psi : \mathcal{M} \times_{ \mathcal{O}} \mathcal{Y} \rightarrow \mathcal{M} \times_{ \mathcal{O}} \mathcal{Y}$ where $c$ is the inverse in the groupoid,
\item composition: $ \circ \times \mathcal{Y} : \mathcal{M} \times_{ \mathcal{O}} \mathcal{Y} \times_{ \mathcal{Y}} \mathcal{M} \times_{ \mathcal{O}} \mathcal{Y} \rightarrow \mathcal{M} \times_{ \mathcal{O}} \mathcal{Y}$.
\end{enumerate}
\end{pro}

\begin{proof}
It suffices to check the axioms of a groupoid one by one.
\end{proof}

\begin{lemma} \label{lemma_comod}
The algebra $H \mf_{ \star}^{\Z/2}$ is a $(H_{ \geq}, \ste_{ \geq})$-comodule algebra, via the left unit, {\em i.e.} $ \psi(a) = a \otimes 1$, $ \psi( \sigma^{-1}) = \sigma^{-1} \otimes 1 + a \otimes \tau_0$, and $ \psi( \sigma^{2}) = \sigma^2 \otimes 1$.
\end{lemma}

\begin{proof}
This is an immediate consequence of the formulae defining the product and coproduct in $(H_{ \geq}, \ste_{ \geq})$.
\end{proof}

\begin{pro} \label{proposition_perturbation_steenrod}
The Hopf algebroid $( H_{ \geq}, \ste_{ \geq})$ satisfies the following properties.
\begin{enumerate}
\item It is a deformation of the non equivariant dual modulo 2 Steenrod algebra, with $ \frak{m} = (a, 1- \sigma^{-1})$ \\
\item The Hopf algebroid $( H \mf_{ \star}^{\Z/2}, \ste_{ \star})$ is isomorphic to $ H \mf_{ \star}^{\Z/2} \otimes_{ H_{ \geq}} ( H_{ \geq}, \ste_{ \geq})$.
\end{enumerate}
\end{pro}

\begin{proof}
 For the first point, observe that by definition of the structure morphisms:
\begin{itemize}
\item on objects,
$$  H_{ \geq} / \frak{m} = \F$$
\item on morphisms, for the generators as $H_{ \geq}$-algebra,
\begin{align*}
\ste_{ \geq} / \frak{m} \rightarrow & \ste_* \\
\xi_i \mapsto & \xi_i^2 \\
\tau_i \mapsto & \xi_i 
 \end{align*}
\end{itemize}
For the second point, since a Hopf algebroid is a cogroupoid object in the category of algebras, by naturality of the construction of Proposition \ref{pro_groupoidgeoffrey} and because a comodule algebra corepresents a functor whose values are left modules over the groupoid corresponding to the Hopf algebroid,  Proposition \ref{pro_groupoidgeoffrey} provides a Hopf algebroid structure on $ H \mf_{ \star}^{\Z/2} \otimes_{ H_{ \geq}} ( H_{ \geq}, \ste_{ \geq}) = ( H \mf_{ \star}^{\Z/2} \otimes_{ H_{ \geq}}H_{ \geq}, H \mf_{ \star}^{\Z/2} \otimes_{ H_{ \geq}} \ste_{ \geq})$ by the $(H_{ \geq}, \ste_{ \geq})$-comodule algebra structure on  $H \mf_{ \star}^{\Z/2}$ of Lemma \ref{lemma_comod}. The assertion is then a consequence of the definitions.
\end{proof}

\subsection{Cofreeness}

In this subsection, all algebras and Hopf algebroids are implicitly finite dimensional $\F$-vector spaces.

The goal of this subsection is to understand when a $RO(\Z/2)$-graded quotient Hopf algebroid of $(H, A)$ satisfies that $(H, A)$ is free as a comodule over it.

\begin{rk}
In the non-equivariant case, $ \ste_*$ is cofree as a comodule over all its quotients Hopf algebras (see \cite[Proposition 4.4]{MM65}). For Hopf algebroids, the situation is more tricky. We will use a modified version of \cite[Theorem A.1.1.17]{Ra86} (\textit{Comodule Algebra Structure Theorem}) to understand the cofreeness of a $RO(\Z/2)$-graded quotient Hopf algebroid over its quotients.
\end{rk}

We now study aspects of deformation theory for Hopf algebroids (we make a slight abuse of terminology by talking about deformation in the context of augmented graded rings, and not of complete local rings).
The motivation comes from Proposition \ref{pro_comparaisonste}, and the previous remark.

\begin{de}
 A deformation of a $RO(\Z/2)$-graded connected Hopf algebroid $(H,A)$ is a $RO(\Z/2)$-graded connected Hopf algebroid $( \tilde{H}, \tilde{A})$ together with an ideal $ \frak m \subset \tilde{H}$ concentrated in degrees of the form $*+n \alpha$ for $n$ non negative, and an identification $( \tilde{H}/ \frak m, \tilde{A}/ \tilde{A} \eta_L (\frak m)) \cong (H,A)$. More precisely:
\begin{enumerate}
 \item $ \tilde{H}/ \frak{m} \cong H$
 \item via the unit $ \eta_L$, we can consider the left ideal generated by $ \frak{m}$ in $ \tilde{A}$. Suppose that $ \tilde{A}/ \tilde{A} \eta_L( \frak{m}) \cong A$
 \item these two isomorphisms induce an isomorphism of Hopf algebroid.
\end{enumerate}
\end{de}

\begin{thm} \label{resultat_coliberte}
Let $( \tilde{H}, \tilde{A})$ be a $RO(\Z/2)$-graded connected Hopf algebroid and $ \tilde{ \mathcal{I}_B} \subset \tilde{ \mathcal{I}_C}$ two $RO(\Z/2)$-graded Hopf algebroid ideals of $( \tilde{H}, \tilde{A})$. We use the following notations:
\begin{align*}
  \tilde{B} = & \tilde{A} / \tilde{ \mathcal{ I}_B} \\
  \tilde{ C} = & \tilde{A} / \tilde{ \mathcal{ I}_C} \\
  \mathcal{ I}_B = & \tilde{ \mathcal{I}_B} / \tilde{ \mathcal{I}_B} \cap \tilde{A} \eta_L( \frak{m}) \\
  \mathcal{ I}_C = & \tilde{ \mathcal{I}_C} /  \tilde{ \mathcal{I}_C} \cap \tilde{A} \eta_L( \frak{m}) \\
  B = & A / \mathcal{ I}_B \\
  C = & A / \mathcal{ I}_C.
\end{align*}
Suppose moreover that
\begin{enumerate}
 \item The ideals $\mathcal{ I}_B$ and $\mathcal{ I}_C$ are Hopf algebroid ideals of $A$. Thus $(H,B)$ and $(H,C)$ are Hopf algebroids, and there are Hopf algebroid morphisms $A \rightarrow B$ and $A \rightarrow C$.
 \item The map $B = A / \mathcal{ I}_B \rightarrow A / \mathcal{ I}_C = C$ makes $B$ a cofree $C$-comodule.
\item $ \tilde{B}$ and $ \tilde{C}$ are free $H$-modules.
\end{enumerate}
Then, $ \tilde{B}$ is a cofree $ \tilde{C}$-comodule.
\end{thm}

\begin{proof}
This proof is an appropriately modified version of \cite[Theorem A1.1.17]{Ra86}. \\
The $C$-comodule $B$ is cofree, so let $ \{1, b_i \}$ be a basis of $B$ as a counital $C$-comodule and denote $ \psi : B \rightarrow C \otimes_H H \{ 1,b_i \}$ the $C$-comodule isomorphism such that $B \surj C$ is the $C$-comodule map such that $ b_i \mapsto 0$. \\

By hypothesis, $ \tilde{C}$ is a free $H$-module. Choose a basis $ \{1,y_i \}$ of $ \tilde{C}$ as a $ \tilde{H}$-module, where $1$ comes from $ \eta_L$. The set $ \{x_i \} = \{1,y_j \} \times  \{1,b_k \}$ is then a basis of $ \tilde{B}$ as a $ \tilde{H}$-module.
 
We also have the following identification
\begin{eqnarray*}
 B &  \cong & \tilde{B}/ \tilde{B} ( \frak{m}) \\
& \cong & \tilde{H} \{x_i \} / ( ( \frak{m})\otimes_{\F} \F \{x_i \}) \\
& \cong & \tilde{H}  / ( ( \frak{m})) \otimes_{\F} \F \{x_i \} \\
&  \cong & H   \{x_i \}.
\end{eqnarray*}
thus $ \{x_i \}$ is also a $H$-module basis for $B$.
Finally, the $ \tilde{H}$-module morphism
$$ \tilde{ \psi} : \tilde{B} \rightarrow \tilde{H} \{ 1,b_i \} $$
induced by $y_i \mapsto 0$ extends $ ( \epsilon \otimes 1) \circ \psi : B \rightarrow H \{ 1,b_i \}$. \\

Let  $ \tilde{ \phi}$ be the graded $ \tilde{ H}$-module morphism defined as the composite
$$ \tilde{B} \stackrel{ \Delta}{ \rightarrow}  \tilde{B} \otimes_{ \tilde{H}} \tilde{B} \surj \tilde{C} \otimes_{ \tilde{H}} \tilde{B} \stackrel{ id \otimes \tilde{ \psi}}{ \longrightarrow} \tilde{C} \otimes_{ \tilde{H}} \tilde{H} \{ 1,b_i \} . $$

And define $ \phi$ as the $C$-comodule morphism
$$ \phi :  B \stackrel{ \Delta}{ \rightarrow} B \otimes_{ H} B \surj C \otimes_{ H} B \xrightarrow{ id \otimes (( \epsilon \otimes id) \circ \psi)} C \otimes_{ H} H \{ b_i \}. $$ The map $ \phi$ is an isomorphism. Moreover, $ \psi$ is a $C$-comodule isomorphism, and projection is compatible with the coproduct on $$ B \rightarrow B \otimes_{H} B$$ and $$ C \rightarrow C \otimes_H C,$$ so $ \phi$ is also a $C$-comodule isomorphism. \\

We now show that $ \tilde{ \phi}$ is an isomorphism. The strategy goes as follows:  we define two decreasing filtrations on $ \tilde{B}$ and $ \tilde{C} \otimes_{ \tilde{H}} \tilde{H} \{ b_i \}$ which are compatible with the map $ \tilde{ \phi}$. We then show that $  \tilde{ \phi}$ induces an isomorphism on the graded object associated with this filtration.

The filtrations are defined on $\tilde{C} \otimes_{ \tilde{H}} \tilde{H} \{ b_i \}$ and $ \tilde{B}$ by:
\begin{align}
 F^d( \tilde{C} \otimes_{ \tilde{H}} \tilde{H} \{ b_i \})= & \langle c \otimes hb_i | deg(hb_i) \geq d \rangle \\
 F^d(\tilde{B})= & \widetilde{ \phi}^{-1}(F^d(\tilde{C} \otimes_{ \tilde{H}} \tilde{H} \{ b_i \}))
\end{align}
where $ \langle S \rangle$, for a set $S$ stands for the cofree $ \tilde{C}$-comodule on $S$.
We check that $ \tilde{ \phi}$ is an isomorphism on the graded object associated with this filtration. \\

Write $ \Delta(b) = \sum b' \otimes b''$ for $b$ in $ \tilde{B}$. By definition of $ \tilde{ \phi}$, we have
 $$ \tilde{ \phi}(b) = \sum [b'] \otimes \tilde{ \psi}(b'').$$

So, let $ \tilde{b}$ be a homogeneous element of $ \tilde{B}$, and write $ \tilde{B}$ in the basis $ \{x_i \}$
$$ \tilde{b} = \sum h_ix_i.$$ 

Consider first an element of the form $h_i x_i$. By definition of Hopf algebroid deformation, $ \frak{m}$ is concentrated in degrees of non negative \textit{twists}, thus, the equality $ \Delta_{ \tilde{B}} = \Delta_B \text{ modulo } \frak{m}$ gives
$$ \Delta_{ \tilde{B}} (h_ib_i) \equiv h_i \Delta_{ \tilde{B}}(b_i) \equiv h_i \Delta_B(b_i) \text{ (mod terms of lower twist)};$$ 
thus, on the graded object associated to the filtration $F^{ \bullet}$, the morphism $ \tilde{ \phi}$ coincides with the unique $ \tilde{H}$-module morphism which extends $ \phi$.

Define now a a filtration  $F^{ \bullet}$ on $C \otimes_{ H} H \{ b_i \}$ by $$F^d( C \otimes_{ H} H \{ b_i \})=  < c \otimes hb_i | deg(hb_i) \geq d >$$ and one on $B$ by pulling back along $ \phi$.

The map $ \phi$ being an isomorphism, it induces an isomorphism on the associated graded $C \otimes_{ H} H \{ b_i \}$ by the filtration $F^d$. But $ \tilde{ \phi}$ coincides with $ \phi$ on the graded object associated to $F^{ \bullet}$. We conclude that $ \tilde{ \phi}$ induces an isomorphism on the graded object associated with the  filtration $F^{ \bullet}$.
Finally, the filtration is exhaustive, and the finite dimensional hypothesis in each degree allows us to apply Mittag-Leffler criterion to conclude that the filtration is also complete.
Consequently, $ \tilde{ \phi}$ is a $ \tilde{C}$-comodule isomorphism.
\end{proof}

\begin{thm} \label{thm_cofree}
Let $B$ and $C$ be two quotient $RO(\Z/2)$-graded Hopf algebroids of $( H \mf_{ \star}^{\Z/2}, \ste_{ \star})$ defined by two profile functions $(h_B,k_B) \geq (h_C,k_C)$. Then there is a $RO(\Z/2)$-graded Hopf algebroid map $B \surj C$, and the induced $C$-comodule structure on $B$ is cofree.
\end{thm}

\begin{proof}
We use that the non-equivariant modulo 2 dual Steenrod algebra is cofree over all its Hopf algebra quotients, thus the hypothesis of Proposition \ref{resultat_coliberte} are satisfied for the deformation given in Proposition \ref{proposition_perturbation_steenrod}. The result is then a consequence of Theorem \ref{resultat_coliberte}.
\end{proof}

\begin{corr} \label{cor_free}
For all $n \geq 0$, the Hopf algebroid $( H \mf_{ \star}^{\Z/2}, \ste_{ \star})$ is cofree as a comodule over its quotients
$ \mathcal{E}(n)_{ \star}$ and $ \ste(n)_{ \star}$, and these quotient Hopf algebroids are cofree over one another.
\end{corr}

\begin{thm} \label{main_thm}
For all $n \geq 0$, the $\Z/2$-equivariant Steenrod algebra $H \mf^{ \star}_{\Z/2}H\F$ is free as a module over
$ \mathcal{E}(n)^{ \star}$ and $ \ste(n)^{ \star}$, and these algebras are free over one another.
\end{thm}

\begin{proof}
Because of Corollary \ref{corr_hypolibertehfd}, Proposition \ref{pro_dualite_operationcoop} is satisfied for $E=H\mf$. Thus, Boardman's duality theorem (cf Proposition \ref{pro_thmboardman}) provides an equivalence between modules over the $\Z/2$-equivariant Steenrod algebra and comodules over $(H\mf_{\star}^{\Z/2}, \ste_{\star})$.
This is now a consequence of Corollary \ref{cor_free}.
\end{proof}

As a consequence of this result, we have an easy computation in K-theory with reality. We first recall a possible construction of this $\Z/2$-equivariant cohomology theory, following Atiyah \cite{At66}

\begin{de}
Let $X$ be a $\Z/2$-space. A Real vector bundle over $X$ is a complex vector bundle over $X$ such that the action of $\Z/2$ is anti-linear on the fibers.
\end{de}

The following proposition is due to Atiyah \cite{At66}.

\begin{pro}
The functor $K\mathbb{R}_0$, defined on objects by $K \mathbb{R}^0(X) = Gr \left(Vect_\mathbb{R}(X) \right)$ extends to a cohomology theory. Denote by $K \mathbb{R}$ the corresponding $\Z/2$-spectrum.
\end{pro}

\begin{corr}[Corollary of Theorem \ref{main_thm}]
One has an isomorphism of $\F[a]$-modules
$$ H\mf^{\star}_{\Z/2} k\R \cong \ste^{\star} / \mathcal{E}^{\star}(2).$$
\end{corr}

\begin{proof}
The proof is similar to the computation of $H\F^*(ku)$, in two step. 
\begin{enumerate}
\item First compute $H\F^*(H\Z)$ by the Bockstein spectral sequence associated to the exact couple $$ \ste^{\star} \rightarrow H\mf^{\star}_{\Z/2}(H\underline{\Z}) \stackrel{2}{\rightarrow} H\mf^{\star}_{\Z/2}(H\underline{\Z}) $$
which collapses by freeness of $\ste^{\star}$ over $\mathcal{E}(1)^{ \star}$, giving $H\mf^{\star}_{\Z/2}(H\underline{\Z}) \cong \ste^{\star} / \mathcal{E}(1)^{\star}$.
\item Then the Bockstein spectral sequence associated to the exact couple $$ H\mf^{\star}_{\Z/2}(H\underline{\Z}) \rightarrow H\mf^{\star}_{\Z/2}(k\R) \stackrel{v_1}{\rightarrow} H\mf^{\star}_{\Z/2}(k\R) $$
collapses by freeness of $\ste^{\star}$ over $\mathcal{E}(2)^{ \star}$. The result follows.
\end{enumerate}
\end{proof}

\begin{rk}
Another use of Theorem \ref{thm_cofree} is to provide change of rings isomorphisms for the $E^2$ page of the $\Z/2$-equivariant Adams spectral sequence (cf \cite[Corollary 6.47]{HK01}) which converges to cohomology with respect to spectra whose homology is of the form $\ste_{\star} // \mathcal{B}$, for $\mathcal{B} = \mathcal{A}(n)_{\star}$ or $\mathcal{E}(n)_{\star}$.
\end{rk}

\bibliographystyle{acm}
\bibliography{biblio}

\end{document}